\documentclass[12pt,letterpaper]{amsart}
\usepackage{euler, epic,eepic,latexsym, amssymb, amscd, amsfonts, xypic, url, color, epsfig}
\input xy
\xyoption{all}

\ProvideTextCommandDefault{\cprime}{\tprime}
 \newlength{\baseunit}               % the basic unit length
 \newcount{\numlines}                % depth of picture (in number of lines)
\theoremstyle{definition}  %Changes theorem style from italics to roman.

\newtheorem{tm}{Theorem}
\newtheorem{pr}[tm]{Proposition}
\newtheorem{lm}[tm]{Lemma}
\newtheorem{co}[tm]{Corollary}
\newtheorem{df}[tm]{Definition}

\newtheorem{rmk}[tm]{Remark}
\newtheorem{example}[tm]{Example}

\newcommand{\KMW}{\mathrm{K}^{\mathrm{MW}}}
\newcommand{\GW}{\mathrm{GW}}

\newcommand{\Z}{\mathbb{Z}}

\newcommand{\kbar}{\overline{k}}
\newcommand{\R}{\mathbb{R}}

\newcommand{\Hom}{\operatorname{Hom}} % Abelian group of homomorphisms
\newcommand{\Gal}{\operatorname{Gal}}

\newcommand\bbb[1]{\ensuremath{{\mathbf{#1}}}}

\newcommand{\Tr}{\operatorname{Tr}}
\newcommand{\ind}{\operatorname{ind}}

\newcommand{\Top}{\mathrm{top}}
\newcommand{\CW}{\widetilde{\operatorname{CH}}}

\newcommand{\PGL}{\operatorname{PGL}}

\newcommand{\GL}{\operatorname{GL}}
\newcommand{\SL}{\operatorname{SL}}
\newcommand{\Gr}{\operatorname{Gr}}

\newcommand{\chr}{\operatorname{char}}

\newcommand{\Proj}{\operatorname{Proj}}

\newcommand{\hidden}[1]{\footnote{Hidden:  #1}}
\renewcommand{\hidden}[1]{}

\newcommand{\bbA}{\mathbf{A}}

\newcommand{\bbP}{\mathbf{P}}
\newcommand{\bbQ}{\mathbf{Q}}
\newcommand{\bbR}{\mathbf{R}}

\newcommand{\bbZ}{\mathbf{Z}}

\newcommand{\calD}{{ \mathcal D}}
\newcommand{\calE}{{ \mathcal E}}

\newcommand{\calL}{{ \mathcal L}}

\newcommand{\calO}{{ \mathcal O}}

\newcommand{\calQ}{{ \mathcal Q}}

\newcommand{\calS}{{ \mathcal S}}
\newcommand{\calT}{{ \mathcal T}}

\newcommand{\calV}{{ \mathcal V}}

\newcommand{\Spec}{\operatorname{Spec}}
\newcommand{\ShHom}{\mathscr{H}\kern -.5pt om}

\begin{document}
\pagestyle{plain}
\title{An arithmetic count of the lines on a smooth cubic surface}

\author{Jesse Leo Kass}
\address{Current: J.~L.~Kass, Dept.~of Mathematics, University of South Carolina, 1523 Greene Street, Columbia, SC 29208, United States of America}
\email{jkass@umich.edu}
%\urladdr{http://people.math.sc.edu/kassj/}

\author{Kirsten Wickelgren}
\address{Current: K.~Wickelgren, Dept.~of Mathematics, Duke University, Durham, NC 27708, United States of America}
\email{kirsten.wickelgren@duke.edu}
%\urladdr{http://people.math.gatech.edu/~kwickelgren3/}

%\address{?}
%\email{?}
\date{\today}
\begin{abstract}
We give an arithmetic count of the lines on a smooth cubic surface over an arbitrary field $k$, generalizing the counts that over $\bbb{C}$ there are $27$ lines, and over $\bbb{R}$ the number of hyperbolic lines minus the number of elliptic lines is $3$. In general, the lines are defined over a field extension $L$ and have an associated arithmetic type $\alpha$ in $L^*/(L^*)^2$. There is an equality in the Grothendieck-Witt group $\operatorname{GW}(k)$ of $k$ $$\sum_{\text{lines}} \Tr_{L/k} \langle \alpha \rangle = 
				15 \cdot \langle 1 \rangle + 12 \cdot \langle -1 \rangle,
$$ where $\Tr_{L/k}$ denotes the trace $\operatorname{GW}(L) \to \operatorname{GW}(k)$. Taking the rank and signature recovers the results over $\bbb{C}$ and $\bbb{R}$. To do this, we develop an elementary theory of the Euler number in $\bbA^1$-homotopy theory for algebraic vector bundles. We expect that further arithmetic counts generalizing enumerative results in complex and real algebraic geometry can be obtained with similar methods.			
\end{abstract}
\maketitle
%\tableofcontents

{\parskip=12pt % closing bracket is just before the bibliography 

\section{Introduction}

In this paper we give an arithmetic count of the lines on a smooth cubic surface in projective space $\bbP_k^{3}$. A celebrated 19th century result of Salmon and Cayley \cite{salmon49} is that: \begin{equation} \label{Eqn: ComplexLineCount}
	\#\text{complex lines on $V$} = 27,
\end{equation}
where $V$ is such a surface over the complex numbers $\bbb{C}$. In particular, this number is independent of the choice of $V$. By contrast, a real smooth cubic surface can contain $3$, $7$, $15$, or $27$ real lines.  

It is a beautiful observation of Finashin--Kharlamov \cite{finashin13} and Okonek--Teleman \cite{okonek14} that while the number of real lines on a smooth cubic surface depends on the surface, a certain signed count of lines is independent of the choice. Namely, the residual intersections of $V$ with the hyperplanes containing $\ell$ are conic curves that determine an involution of $\ell$, defined so that two points are exchanged if they lie on a common conic.  Lines are classified as either hyperbolic or elliptic according to whether the involution is hyperbolic or elliptic as an element of $\PGL_2$ (i.e. whether the fixed points are defined over $k$ or not).  Finashin--Kharlamov and Okonek--Teleman observed that the equality 
\begin{equation} \label{Eqn: RealLineCount}
	\#\text{real hyperbolic lines on $V$} - \#\text{real  elliptic lines on $V$} = 3.
\end{equation}
can be deduced from Segre's work.  They gave new proofs of the result and extended it to more general results about linear subspaces of hypersurfaces. We review this and later work below. 

We generalize these results to an arbitrary field $k$ of characteristic not equal to 2.  The result is particularly simple to state when $k$ is a finite field $\mathbb{F}_{q}$.  As was observed in \cite[pages~196--197]{hirschfeld85}, a line $\ell \subset V$ admits a distinguished involution, just as in the real case, and we classify $\ell$ as either hyperbolic or elliptic using the involution, as before. 

When all $27$ lines on $V$ are defined over $\mathbb{F}_{q}$, we prove
\[
	\#\text{elliptic lines on $V$} = 0 \text{ mod 2.}
\]
For $V$ an arbitrary smooth cubic surface over $\mathbb{F}_{q}$, we have
\begin{tm} \label{Thm: FiniteFieldLineCount}
	The lines on a smooth cubic surface $V \subset \bbP^{3}_{\mathbb{F}_{q}}$ satisfy
	\begin{multline} \label{Eqn: FiniteFieldLineCount}
		\#\text{elliptic lines on $V$ with field of definition $\mathbb{F}_{q^a}$ for $a$ odd} \\+  \#\text{hyperbolic lines on $V$ with field of definition 
		$\mathbb{F}_{q^a}$ for $a$ even} = 0 \text{ (mod 2).}
	\end{multline}
\end{tm}

Here a {\em line} means a closed point in the Grassmannian of lines in $\bbP^3_k$, so a line corresponds to a Galois orbit of lines over an algebraic closure. For example, consider the Fermat surface $V = \{ x_1^3+x_2^3+x_3^3+x_4^3=0 \}$ over $\mathbb{F}_{q}$ of characteristic $p \ne 2, 3$.  When $\mathbb{F}_{q}$ contains a primitive third root of unity $\zeta_{3}$, all the $27$ lines are defined over $\mathbb{F}_{q}$ and are hyperbolic.  Otherwise $V$ contains $3$ hyperbolic lines defined over $\mathbb{F}_{q}$ and $12$ hyperbolic lines defined over $\mathbb{F}_{q^2}$. (See the Notation and Conventions Section \ref{Section: LineConventions} for further discussion.)

For arbitrary $k$, we replace the signed count valued in $\bbb{Z}$ with a count valued in the Grothendieck--Witt group $\operatorname{GW}(k)$ of nondegenerate symmetric bilinear forms. (See \cite{lam05} or \cite{morel} for information on $\operatorname{GW}(k)$.) The signs are replaced by classes $\langle a \rangle$ in $\GW(k)$ represented by the bilinear pairing  on $k$ defined $(x, y) \mapsto a x y$ for $a$ in $k^*$. The class $\langle a \rangle$ is determined by arithmetic properties of the line, namely its field of definition and the associated involution, and we call this class the {\em type} of the line.  

As we will discuss later, the reason for enumerating lines as elements of $\operatorname{GW}(k)$ is that the Grothendieck--Witt group is the target of Morel's degree map in $\bbA^{1}$-homotopy theory. The types $\langle a \rangle$ are local contributions to an Euler number. Despite this underlying reason, the calculation of the types $\langle a \rangle $ as well as the proof of the arithmetic count are carried out in an elementary manner and without direct reference to $\bbA^{1}$-homotopy theory.  

Theorem \ref{Thm: FiniteFieldLineCount} is a special case of the following more general result. Let $\ell$ be a line on $V$ and let $k \subseteq L$ denote the field of definition of $\ell$, which must be separable (Corollary \ref{co:line_field_def_is_separable}). There is a transfer or trace map $\operatorname{Tr}_{L/k}: \GW(L) \to \GW(k)$ defined by taking a bilinear form $\beta: A \times A \to L$ on an $L$-vector space $A$ to the composition $\Tr_{L/k} \circ \beta: A \times A \to L \to k$ of $\beta$ with the field trace $\Tr_{L/k}: L \to k$, the vector space $A$ now being viewed as a vector space over $k$. 

We refine the classification of lines on $V$ as either hyperbolic or elliptic as follows.  Define the type of an elliptic line $\ell$ with field of definition $L$ to be the class $\mathcal{D} \in L^{\ast}/(L^{\ast})^2$ of the discriminant of the fixed locus, i.e., to be the $\mathcal{D}$ such that $L(\sqrt\mathcal{D})$ is the field of definition of the  fixed locus. We extend this definition by defining the type of a hyperbolic line to be $1$ in $L^{\ast}/(L^{\ast})^2$.  Observe that when $k=\mathbb{R}$ (respectively, $\mathbb{F}_{q}$), the type of an elliptic line is $-1$ (respectively, the unique non-square class), but in general there are more possibilities.  The type can also be interpreted in terms of the $\bbA^{1}$-degree of the involution.  We explain this and other aspects of hyperbolic and elliptic lines in Section~\ref{Section: ClassificationOfLines}.

With this definition of type, we now state the main theorem.
\begin{tm}[Main Theorem] \label{Theorem: MainTheorem}
	The lines on a smooth cubic surface $V \subset \bbP^{3}_{k}$ satisfy 
	\begin{equation} \label{Eqn: FinalLineCount}
		 \sum_{h \in L^{\ast}/(L^{\ast})^{2}} \left( \#\text{lines of type $h$}\right) \cdot \operatorname{Tr}_{L/k}( \langle h \rangle)    = 15 \cdot \langle 1 \rangle + 12 \cdot \langle -1 \rangle.
	\end{equation}
\end{tm}
From Equation~\eqref{Eqn: FinalLineCount}, we recover the complex count  \eqref{Eqn: ComplexLineCount} by taking the rank of both sides, the real count \eqref{Eqn: RealLineCount} by taking the signature (the complex lines contribute the signature zero class $\langle 1 \rangle + \langle -1 \rangle$), and the finite field count \eqref{Eqn: FiniteFieldLineCount} by taking the discriminant (which is an element of $\mathbb{F}_{q}^{\ast}/(\mathbb{F}_{q}^{\ast})^{2} = \mathbb{Z}/2$). 

Over a more general field, one gets analogues of those equations by taking the rank, discriminant, or signature (with respect to an ordering) of \eqref{Eqn: FinalLineCount}, but there can be more subtle constraints as well.  For example, there is no smooth cubic surface over $k=\bbQ$ with 27 lines defined over $k$, two with type $\langle 3 \rangle$, thirteen  with type $\langle 1 \rangle $, and twelve with type $\langle -1 \rangle$.  Indeed, this is a special case of  Theorem~\ref{Theorem: MainTheorem} because  $2 \cdot \langle 3 \rangle + 13 \cdot \langle 1 \rangle + 11 \cdot \langle -1 \rangle$ and $15 \cdot \langle 1 \rangle + 12 \cdot \langle -1 \rangle$ have different Hasse--Witt invariants at the prime $p=3$.  We cannot, however, rule out the existence of such a surface  using the analogues of equations \eqref{Eqn: ComplexLineCount}, \eqref{Eqn: RealLineCount}, and \eqref{Eqn: FiniteFieldLineCount} because the forms have the same rank, discriminant, and signature.

Theorem~\ref{Theorem: MainTheorem} only applies to a smooth cubic surface, but we discuss the singular surfaces at the end of this section, right before Section~\ref{Subsect: RelationToOther}.

The statement and proof of Theorem~\ref{Theorem: MainTheorem} are inspired by Finashin--Kharlamov and Okonek--Teleman's proof of the real line count \eqref{Eqn: RealLineCount}, which in turn is inspired by a proof of the complex line count \eqref{Eqn: ComplexLineCount} that runs as follows. Let $\mathcal{S}$ denote the tautological subbundle on the Grassmannian $G := \operatorname{Gr}(4, 2)$ of $2$-dimensional subspaces of the $4$-dimensional vector space $k^{\oplus 4}$.  Given an equation $f \in \mathbb{C}[x_1, x_2, x_3, x_4]$ for a complex cubic surface $V\subset \bbP^{3}_{\mathbb{C}}$, the rule
\[
	\sigma_{f}(S) = f|S
\] 
determines a section $\sigma_{f}$ of the vector bundle $\mathcal{E} = \operatorname{Sym}^{3}( \mathcal{S}^{\vee})$.  By construction, the zeros of $\sigma_{f}$ are the lines contained in $S$, but a local computations shows that, when $S$ is smooth, the section $\sigma_{f}$ has only simple zeros, and so in this case, the count of zeros equals the Chern number $c_{4}(\mathcal{E})$.  We immediately deduce that the number of lines on a complex smooth cubic surface is independent of the surface, and it can be shown that this independent count is 27 by computing $c_{4}(\mathcal{E})$ using structural results about e.g.~the cohomology of the complex Grassmannian.

The proofs by Finashin--Kharlamov and Okonek--Teleman of the real count of lines are similar to the proof of the complex count just given.  Both the real Grassmannian $G(\bbR)$ and the vector bundle $\mathcal{E}$ are orientable, so after fixing orientations, the Euler number $e(\mathcal{E})$ of $\mathcal{E}$ is well-defined.  The real count \eqref{Eqn: RealLineCount} can be proven along the same lines as the complex count, only with Chern number replaced by the Euler number.  One new complication is that, in addition to showing that $\sigma_{f}$ has only simple zeros,  it is necessary to also show that the local index of a zero is $+1$ at a hyperbolic line and $-1$ at an elliptic line.

For this proof to generalize to a count over an arbitrary field, we need a generalization of the Euler number, which is furthermore computable as a sum of local indices. Classically, the local index of an isolated zero of a section $\sigma$ can be computed by choosing local coordinates and a local trivialization, thereby expressing $\sigma$ as a function $\bbb{R}^r \to \bbb{R}^r$ with an isolated zero at the origin. The local degree of this function is then the local index, assuming that the choice of local coordinates and trivialization were compatible with a given orientation or relative orientation. In \cite{eisenbud78}, Eisenbud suggested defining the local degree of a function $\bbA_k^r \to \bbA_k^r$ to be the isomorphism class of the bilinear form now appearing in the Eisenbud--Khimshiashvili--(Harold) Levine signature formula. This bilinear form is furthermore explicitly computable by elementary means. For example, if the Jacobian determinant $J$ is non-zero at a point with residue field $L$, then the local degree is $\Tr_{L/k} \langle J \rangle$. In \cite{KWA1degree}, we showed it is also the local degree in $\bbA^1$-homotopy theory. 

Define the Euler number $e(\mathcal{E}) \in \operatorname{GW}(k)$ to be the sum of the local indices using the described recipe and this local degree. Since local coordinates are not as well-behaved for smooth schemes as for manifolds, some finite determinacy results are being used implicitly, but in the present case, these are elementary algebra. We show the Euler number is well-defined using Scheja--Stoch's perspective on duality for finite complete intersections (e.g., \cite{scheja}), which shows that this local degree behaves well in families. 

The result is as follows. Let $X$ be a smooth scheme of dimension $r$ over $k$. Let $\mathcal{E} \to X$ be a relatively oriented rank $r$ vector bundle such that any pair of sufficiently general sections can be connected by sections with only isolated zeros (as in Definition \ref{df:ss'connected_isolated_zeros}), potentially after further extensions of odd degree.

\begin{tm}
The Euler number $$e(\mathcal{E}) = \sum_{p \text{ such that } \sigma(p) = 0} \ind_p \sigma$$ is independent of the choice of section $\sigma$.
\end{tm}

This is shown as Corollary \ref{e(E)well-defined} in Section \ref{section:Euler_number_relatively_oriented_VB}, and some examples are computed. We deduce that the left-hand side of \eqref{Eqn: FinalLineCount} is independent of the choice of surface.  We then show that this common class equals the right-hand side \eqref{Eqn: FinalLineCount} by evaluating the count on a specially chosen smooth surface. 

We remark that for $f$ defining a smooth cubic surface, the corresponding section $\sigma_f$ has only simple zeros (the Jacobian determinant is non-zero), so the more general calculations of local degree from \cite{scheja} \cite{eisenbud77} \cite{eisenbud78} are only needed here to ensure that the local degree behaves well in families \cite{scheja}.

When $f$ defines a singular cubic surface, the section $\sigma_{f}$ can have nonsimple zeros.  If we additionally assume $\sigma_{f}$ has only isolated zeros (i.e.~the surface contains only finitely many lines), the index of a zero can be computing using the main result of \cite{KWA1degree}.  For example, when $f=x_0^2 x_3 + x_0 x_2^2+x_1^3$ (a surface with an $E_6$-singularity), $\sigma_{f}$ has a unique zero whose local index is described in \cite[Section~7]{KWA1degree}.  If the type of a line corresponding to a nonsimple zero of $\sigma_{f}$ is defined to be the local index, then Theorem~\ref{Theorem: MainTheorem} remains valid for singular cubic surfaces containing only finitely many lines.

\subsection{Relation to other work} \label{Subsect: RelationToOther}
 A large number of Euler classes in arithmetic geometry have been constructed, but the definition used here seems to be original.  Closest to our definition is that of Grigor`ev and Ivanov in \cite{Grig_Ivan}. For a perfect field $k$ of characteristic different from $2$, they consider the quotient $\Delta(k) = \GW(k)/\operatorname{TF}(k)$ of the Grothendieck-Witt group by the subgroup generated by trace forms of field extensions, and define the Euler number to be the element of $\Delta(k)$ given by the sum of the indices of the $k$-rational zeros of a chosen section with isolated zeros. Quotienting by the group generated by the trace forms allows them to ignore the zeros which are not $k$-rational. (See Proposition \ref{index=Trindk(p)}.) Their local index is also inspired by Eisenbud's article \cite{eisenbud78}. By contrast, they only consider orientations in the case of real closed field $k$\hidden{ they do not consider relative orientations at all}, and the group $\Delta(k)$ is quite small, unless $k$ is algebraically closed or real closed: They show that  for $k$ an algebraically closed field, the rank induces an isomorphism $\Delta(k) \cong \Z$;  for $k$ a real closed field, the signature induces an isomorphism $\Delta(k) \cong \Z$; for fields where there is a fixed prime $p$ such that all extensions have degree $p^m$ and which are not algebraically closed or real closed, the rank determines an isomorphism $\Delta(k) \cong \Z/p$, and for all other fields $\Delta(k) = 0$.  For example, $\Delta(k)$ is zero for a finite field or a number field, while the Grothendieck-Witt group of such fields is infinite and contains distinct elements with the same rank. 

In $\bbA^1$-homotopy theory, there is an Euler class in Chow-Witt groups, also called oriented Chow groups, twisted by the dual determinant of the vector bundle, defined by Barge-Morel \cite{BargeMorel} and Fasel \cite{FaselGroupesCW}: A rank $r$ vector bundle $E$ on a smooth $d$-dimensional scheme $X$ gives rise to an Euler class $e(E)$ in $\CW^r(X, \det E^*)$. In \cite[Chapter 8.2]{morel}, Morel defines an Euler class in $H^r(X, \KMW_r (\det E^*))$ as the primary obstruction to the existence of a non-vanishing section. When the $\det E^*$ is trivial, Asok and Fasel used an isomorphism $H^r(X, \KMW_r (\det E^*)) \cong \CW^r(X, \det E^*)$ (see  \cite[Theorem~2.3.4]{AsokFasel_comp_euler_classes}), analogous to Bloch's formula for Chow groups, to show these two Euler classes differ by a unit \cite[Theorem~1]{AsokFasel_comp_euler_classes} provided $k$ is a perfect field  with $\operatorname{char} k \ne 2$.  In a preprint that appeared while this paper was being written, Marc Levine extended this result to the case where the determinant is possibly nontrivial \cite[Proposition~11.6]{Levine-EC}. In the same paper, Levine also developed the properties of the Euler number or class of a relatively oriented vector bundle of rank $r$ on a smooth, proper $r$-dimensional $X$ defined by pushing-forward the Euler class along the map $$\CW^r(X, \omega_{X/k}) \to \CW^0(k) \cong \GW(k),$$ where $\omega_{X/k}$ is the canonical sheaf and where the relative orientation is used to identify $\CW^r(X, \det E^*)$ and $\CW^r(X, \omega_{X/k})$. It is shown that this class coincides with $e(\mathcal{E})$ when our $e(\mathcal{E})$ is defined in \cite[Second Corollary pg 3]{bachmann_wickelgren}. We give here a development of the Euler number in $\GW(k)$ which does not use the machinery of oriented Chow groups, and which is elementary in the sense that it only requires algebra to compute, along with some additional duality theory from commutative algebra to show it is well-defined.

In earlier work, Nori, Bhatwadaker, Mandal, and Sridharan defined Euler class groups and weak Euler class groups for affine schemes. These have been used to study the question of when a projective module splits off a free summand \cite{Mandal-Srid-Euler_classes_complete_intersections} \cite{BhatSrid-Zero_cycles} \cite{BhatSrid-Euler_class_group} \cite{bhatwadekar06}. For smooth, affine varieties these groups can be mapped to the Chow--Witt groups in a way that is compatible with Euler classes under suitable additional hypotheses \cite[Propositions 17.2.10,17.2.11]{FaselGroupesCW}.

The signed count of real lines \eqref{Eqn: RealLineCount} on a cubic surface has been extended in various ways. Benedetti--Silhol gave a topological interpretation of  the classification of lines as hyperbolic or elliptic using pin structures in \cite{benedetti95}.  The signed count can also be identified as an enumerative invariant defined in work of Solomon.  In \cite{solomon06}, Solomon defined open Gromov--Witten invariants for a suitable real symplectic manifold of dimension 4 or 6.  Solomon's invariants count certain real genus $g$ $J$-holomorphic curves, and Finashin--Kharlamov explained in \cite[Section~5.2]{finashin13} that \eqref{Eqn: RealLineCount} equals Solomon's invariant when $g=0$ and the symplectic manifold is the space of real points of the cubic surface.  Solomon, together with Horev, also gave an alternative proof of \eqref{Eqn: RealLineCount} when the cubic surface is the blow-up of the plane in \cite{solomon12}, a paper in which they more generally compute the open Gromov--Witten invariants for certain blow-ups of the plane.

Another  approach to studying the lines on a cubic surface is given by Basu, Lerario, Lundberg, and Peterson in \cite{Basu_Random}.  They analyze the count of lines from the perspective of probability theory  and give a new probabilistic proof of  \eqref{Eqn: RealLineCount} \cite[Proposition~2]{Basu_Random}.

Finashin--Kharlamov and Okonek--Teleman, in \cite{okonek14} and \cite{finashin13}, compute more generally a signed count of the real lines on a hypersurface of degree $2n-3$ in real projective space $\bbP^n_{\R}$.   As in the case of cubic surfaces, this count is obtained as a computation of an Euler class, but unlike the case of cubic surfaces, care must be taken when defining the Euler class because  the relevant real Grassmannian  can be non-orientable, and the Euler number is often only defined for oriented vector bundles on an oriented manifold.  The two sets of researchers address this complication in different ways.  Finashin--Kharlamov work on the orientation cover of $G(\bbR)$ and orient the pullback of $\mathcal{E}$.  By contrast, Okonek--Teleman work on $G(\bbR)$ and construct a suitably defined relative orientation of  $\mathcal{E}$ with an associated Euler class.  These results are also analyzed by Basu, Lerario, Lundberg, and Peterson using tools from probability in \cite{Basu_Random}.  Further extensions of these ideas are found in \cite{finashin15} \cite{OT14-wall-crossing}.

The work described in the present paper is part of a broader program aimed at using $\bbA^1$-homotopy theory to prove arithmetic enrichments of results in enumerative geometry, with earlier results by Marc Hoyois \cite{Hoyois_lef}, the present authors \cite{KWA1degree}, and Marc Levine \cite{Levine-EC, Levine-NormalCone}.  After a version of this paper was posted to the arXiv, additional results were obtained by  Marc Levine \cite{Levine-Witt, Levine-Welschinger},  Stephen McKean \cite{mckean}, Sabrina Pauli \cite{pauli}, Padma Srinivasan with the second author \cite{Srinivasan}, and Matthias Wendt \cite{Wendt-oriented_schubert}. The most closely related of these results is  Levine's \cite{Levine-Witt} and Pauli's  \cite{pauli}.  Levine computes the $\bbb{A}^1$-Euler number of $\operatorname{Sym}^{2n-5} \calS^*$ on $\Gr(n,2)$, and Pauli analyzes the lines on the quintic $3$-fold. 

\subsection{Acknowledgements}
We wish to thank Mohammed Alabbood, Alexey Ananyevskiy, Aravind Asok, Eva Bayer-Fluckiger, Thomas Brazelton, Alex Duncan, J.~W.~P.~Hirschfeld, Hannah Larson, Marc Levine, Stephen McKean, Ivan Panin, Sabrina Pauli and Isabel Vogt for useful comments.

Jesse Leo Kass was partially sponsored by the Simons Foundation under Grant Number 429929, by the National Security Agency under Grant Number H98230-15-1-0264, and by the National Science Foundation under Grant Number DMS-2001565.  The United States Government is authorized to reproduce and distribute reprints notwithstanding any copyright notation herein. This manuscript was submitted for publication with the understanding that the United States Government is authorized to reproduce and distribute reprints.

Kirsten Wickelgren was partially supported by National Science Foundation Award DMS-1406380, DMS-1552730 and DMS-2001890.

\section{Notation and conventions} \label{Section: LineConventions}

Given a $k$-vector space $A$ and an integer $r$, the Grassmannian parameterizing  $r$-dimensional subspaces of $A$ will be denoted by $\operatorname{Gr}(A, r)$.  We will write $\operatorname{Gr}(n, r)$ for $\operatorname{Gr}( k^{\oplus n}, r)$.  $\bbP(A)$ is $\operatorname{Gr}(A, 1)$ or equivalently $\operatorname{Proj}( \operatorname{Sym}(A^{\vee}))$.  With this convention, $H^{0}( \bbP(A), \calO(1)) = A^{\vee}$.  $\bbP^{n}_{k}$ is $\bbP(k^{\oplus n+1})$.  The standard basis of $k^{\oplus 4}$ is $(1, 0, 0, 0), (0, 1, 0, 0), (0, 0, 1, 0), (0, 0, 0, 1)$.  The dual basis of $(k^{\oplus 4})^{\vee}$ is denoted $x_1, x_2, x_3, x_4$.

A linear system on a projective  $k$-variety $V$ is a pair $(T, \calL)$ consisting of a line bundle $\calL$ and a subspace $T \subset H^{0}(V, \calL)$ of the space of global sections.  The linear system $(T, \calL)$ is base-point-free if $\cap_{s \in T} \{ s=0 \}$ is the empty subscheme.  If $(T, \calL)$ is base-point-free, then there is a unique morphism $\pi \colon V \to \bbP(T^{\vee})$ together with an isomorphism $\pi^{*}\calO(1) \cong \calL$ that induces the identity on $T$.

In general, calligraphy font denotes a family of objects, such as $\mathcal{E}$ denoting a vector bundle because it is a family of vector spaces. However, when there is a family of vector bundles, the family then is denoted $\mathcal{E}$.

The concept of a line on a scheme over the possibly non-algebraically closed field $k$ is slightly subtle and plays a fundamental role here.  We use the following.
\begin{df}
	A \textbf{line} $\ell$ in $\bbP^{3}_{k}$ is a closed point of $\operatorname{Gr}(4, 2)$.  The residue field of this closed point is called the  \textbf{field of definition} of $\ell$.
\end{df}
To a line $\ell$ with field of definition $L$, there is the following associated closed subscheme of $\bbP^{3}_{L}$.  The closed point $\ell \in \operatorname{Gr}(4, 2)$ defines a morphism $\Spec(L) \to \operatorname{Gr}(4, 2)$. If the pullback of the tautological subbundle under this morphism is the rank $2$ submodule $S \subset L^{\oplus 4}$, then the homogeneous ideal generated by $\operatorname{ann}(S) \subset \operatorname{Sym} ((L^{4})^{\vee})$ defines a subscheme of $\bbP^{3}_{L}$.  By abuse of notation  we denote this subscheme by $\ell$. 

For $a$ in $k^*$, the element of the Grothendieck--Witt group $\GW(k)$ represented by the symmetric, nondegenerate, rank $1$ bilinear form $(x,y) \mapsto a xy$ for all $x,y$ in $k$ is denoted by $\langle a \rangle$.

\section{Hyperbolic and elliptic lines} \label{Section: ClassificationOfLines}
Here we  define the type of a line on a cubic surface over an arbitrary field $k$ of characteristic $\ne 2$, define hyperbolic and elliptic lines, and  derive an explicit  expression for the type (Proposition~\ref{Prop: ExpressionForType}). This expression will be used in Section~\ref{Section: Counting} to relate the type to a local $\bbA^{1}$-Euler number.

We fix a cubic polynomial $f \in k[x_1, x_2, x_3, x_4]$ that defines a $k$-smooth cubic surface, which we denote by $V := \{ f=0 \} \subset \bbP^{3}_{k}$.

\begin{df}
	Suppose that $\ell$ is a line contained in $V$, with field of definition $L$.  Define $T\subset (L^{\oplus 4})^{\vee} = H^{0}( \bbP_{L}^{3}, \calO(1))$ to be the vector space of linear polynomials that vanish on $\ell$.  This vector space is naturally a subspace of  the space of global sections of $\calO_{V}(1)$ and the space of global sections of the sheaf $I_{\ell}(1) := I_{\ell} \otimes \calO_{V}(1)$ of linear polynomials vanishing on $\ell$.
\end{df}
The subspace $T$ can alternatively be described as  $T = \operatorname{ann}(S)$ for $S \subset L^{\oplus 4}$  the subspace corresponding to $\ell$.

The ideal sheaf $I_{\ell}$  is a line bundle because $\ell$ is a codimension $1$ subscheme of the smooth surface $V \otimes_{k} L$.  Thus $T$ defines two linear systems on $V \otimes_{k} L$: the linear system  $(T, \calO_V(1))$ and the linear system $(T, I_{\ell}(1))$.  The elements $(T, \calO_V(1))$ are the intersections with planes containing $\ell$, while the elements of $(T, I_{\ell}(1))$ are the residual intersections with these planes.  The linear system $(T, \calO_V(1))$ has base-points, namely the points of $\ell$, but as the following lemma shows, the other linear system is base-point-free.

\begin{lm}
	The linear system $(T, I_{\ell}(1))$ is base-point-free.
\end{lm}
\begin{proof}
	The sheaf $I_{\ell}(1)$ is the restriction of the analogous sheaf on $\bbP^{3}_{L}$, and the sheaf on $\bbP^{3}_{L}$ is generated by $T$ by the definition of the subscheme $\ell \subset \bbP^{3}_{L}$ (see Section~\ref{Section: LineConventions}).  We conclude that the same holds on $V$, and $T$ generating $I_{\ell}(1)$ is equivalent to base-point-freeness.
\end{proof}

\begin{df}
	Let $\pi \colon V \otimes L \to \bbP(T^{\vee})$ be the morphism associated to the base-point-free linear system $(T, I_{\ell}(1))$. The restriction of $\pi$ to $\ell$ is a finite morphism of degree $2$ (see Lemma~\ref{Lemma: ResidualPtsOnLine}), hence is Galois (as $\operatorname{char} k \ne 2$).  We denote the nontrivial element of the Galois group of $\ell \to \bbP(T^{\vee})$ by
	\[
		i \colon \ell \to \ell.
	\]
\end{df}
The involution $i$ is the one discussed in the introduction. Concretely $\pi$ is the unique morphism that extends projection from $\ell$. One may identify $ \bbP(T^{\vee})$ with the space of planes in $\bbP^{3}_{L}$ containing $\ell$. Under this identification, $\pi$ sends a point $p$ to the tangent space $T_pV$ to $V$ at $p$ viewed as a projective plane of dimension $2$, and therefore the involution $i$ swaps $p$ and $q$ if and only if $T_p V = T_q V$. Note that the intersection of $T_pV$ with $V$ is a degree $3$ plane curve containing $\ell$, which can therefore be described as a conic $Q$ union $\ell$. We see that $\pi$ should be degree $2$ because given a point $p$ of $\ell$, the intersection of $Q$ with $\ell$ contains two points (counted with multiplicity), and these are precisely the points with the same tangent space.  

\begin{rmk}
	Recall that we require $\operatorname{char} k \ne 2$, and this requirement is important because otherwise the involution $i$ might not exist, in which case the type is undefined.  Indeed, consider the surface $V$ over $\mathbb{F}_{2}$ defined by $f = x_{1}^3+x_{2}^3+x_{3}^3+x_{4}^3$.  This surface contains the line $\ell$ defined by the subspace spanned by $(1, 1, 0, 0)$ and $(0, 0, 1, 1)$.  The morphism $\pi \colon \ell \to \bbP(T^{\vee})$ is purely inseparable, as can be seen either by direct computation or an application of Lemma~\ref{Lemma: ResidualPtsOnLine} below.  In particular, $\ell$ does not admit a nontrivial automorphism that respects $\pi$.
\end{rmk}

Having defined $i$, we can now define hyperbolic and elliptic lines in direct analogy with Segre's definition. We use Morel's $\bbA^{1}$-Brouwer degree, as constructed in \cite{morel}. 

\begin{df} \label{Definition: Type}
	The \textbf{type} of a line on $V$ is $\langle -1 \rangle \cdot \deg^{\bbA^{1}}(i)$, the product of $\langle -1 \rangle$ and $\bbA^{1}$-degree of the associated involution $i$.  We say that the line is \textbf{hyperbolic} if the type equals $\langle 1 \rangle$ (i.e.~$\deg^{\bbA^{1}}(i) = \langle -1 \rangle$).  Otherwise we say that the line is \textbf{elliptic}.
\end{df}

The general definition of the $\bbA^{1}$-degree is complicated, but $\deg^{\bbA^{1}}(i)$ has a simple description.  If we identify $\ell$ with $\bbP_{L}^{1}$ so that $i$ is the linear fractional transformation  $(\alpha z+\beta)/(\gamma z +\delta)$, then the $\bbA^{1}$-degree is $\langle \alpha \delta -  \beta \gamma \rangle \in \operatorname{GW}(k)$.  In particular, $\ell$ is hyperbolic if and only if $-(\alpha \delta -  \beta \gamma)$ is a perfect square in $L$.  

We define the type to be the negative of the degree rather than the degree itself so that, when $k=\mathbb{R}$, the type is consistent  with the sign conventions in  \cite{finashin13, okonek14}.  There hyperbolic lines are counted with sign $+1$ and elliptic lines with sign $-1$.

We now derive an expression for the fibers of $\pi \colon \ell \to \bbP(T^{\vee})$.

\begin{lm} \label{Lemma: ResidualPtsOnLine}
		Suppose that $V$ contains the line $\ell$ defined by the subspace spanned by $(0,0,1,0)$ and $(0, 0, 0, 1)$.  Write
			\[
				f = x_{1} \cdot P_1 + x_{2} \cdot P_2
			\]
		for homogeneous quadratic polynomials $P_1, P_2 \in k[x_1, x_2, x_3, x_4]$
			
		  Then the fiber of $\pi \colon \ell \to \bbP(T^{\vee})$ over the $k$-point corresponding to the $1$-dimensional subspace spanned by $(a, b, 0, 0) \in T^{\vee}$ is 
		\begin{equation} \label{Eqn: ResPoints}
				\{ a \cdot P_{1}(0, 0, x_{3}, x_{4}) + b \cdot P_{2}(0, 0, x_{3}, x_{4})=0 \} \subset \ell
		\end{equation}
\end{lm}
\begin{proof}
The point corresponding to $(a, b, 0, 0)$ is the zero locus of $b x_{1} - a x_{2}$, considered as a global section of $\calO_{\bbP(T^{\vee})}(1)$.  By construction, the preimage of this point under $\pi \colon \ell \to \bbP( T^{\vee})$ is the zero locus of $b x_{1} - a x_{2}$ considered as a global section of $\calO_{\ell} \otimes I_{\ell}(1)$.  We prove the lemma by identifying $\calO_{\ell} \otimes I_{\ell}(1)$ with $\calO_{\ell}(2)$ in such a way that $b x_{1} - a x_{2}$ is identified with the polynomial in \eqref{Eqn: ResPoints}.

Consider the line bundle $I_{\ell}(1)$.  On $V$, we have $0 = x_{1} P_1 + x_{2} P_2$, so $x_{1} = - P_{2}/P_{1} \cdot x_{2}$, showing that $x_{2}$ generates $I_{\ell}(1)$ on $\{ P_1 \ne 0 \}$ and $x_{1}$ generates $I_{\ell}(1)$ on $\{ P_{2} \ne 0 \}$.  We conclude that the analogue is true for $\calO_{\ell} \otimes I_{\ell}(1)$, and the map sending $x_{2}$ to $-P_{1}(0, 0, x_{3}, x_{4})$ and $x_{1}$ to $P_{2}(0, 0, x_{3}, x_{4})$ defines an isomorphism $\calO_{\ell} \otimes I_{\ell}(1) \cong \calO_{\ell}(2)$ that sends $b x_{1} - a x_{2}$ to $a P_{1}(0, 0, x_{3}, x_{4}) + b P_{2}(0, 0, x_{3}, x_{4})$.
\end{proof}

We now collect some general results about involutions and then apply those results to get a convenient expression for $\deg^{\bbA^{1}}(i)$.

\begin{lm}\label{Lemma: Beauville}
	Every nontrivial involution $i \colon \bbP^{1}_{k} \to \bbP^{1}_{k}$ is conjugate to the involution $z \mapsto -\alpha/z$ for some $\alpha \in k$.
\end{lm}
\begin{proof}
	This is \cite[Theorem~4.2]{beauville10}.
\end{proof}

\begin{lm} \label{Lemma: DegreeAndFixedLocusOne}
	The $\bbA^1$-degree of $i(z) = -\alpha/z$ is $\langle \alpha \rangle \in \operatorname{GW}(k)$.
\end{lm}
\begin{proof}
	This is a special case of e.g.~the main result of \cite{cazanavea}.
\end{proof}

\begin{co} \label{Corollary: Beauville}
	If $i$ is a nontrivial involution on $\bbP^{1}_{k}$ and $\mathcal{D} \in k$ is the discriminant of the fixed subscheme of $i$, then 
	\[
	 \langle -1 \rangle \cdot \deg^{\bbA^{1}}(i) = \langle \mathcal{D}  \rangle \text{ in $\operatorname{GW}(k)$.}
	\]
\end{co}
\begin{proof}
	Both the $\bbA^{1}$-degree and the class of the discriminant are unchanged by conjugation, so by Lemma~\ref{Lemma: Beauville}, it is enough to prove result when $i$ is the involution $i(z) = -\alpha/z$.  In this case,  the fixed subscheme is  $\{ z^2+\alpha =0 \}$, which has discriminant $-4 \alpha$.  We have that $\langle -4 \alpha \rangle = \langle -\alpha \rangle$, and the second class is $ \langle -1 \rangle \cdot \deg^{\bbA^{1}}(i)$ by Lemma~\ref{Lemma: DegreeAndFixedLocusOne}.
\end{proof}

\begin{pr} \label{Prop: ExpressionForType}
	Let $e_1, e_2, e_3, e_4$ be a basis for $k^{\oplus 4}$ such that the subspace $S := k \cdot e_3 + k \cdot e_4$ defines a line contained in $V$. Let $x_1, x_2, x_3, x_4$ denote the dual basis to $e_1, e_2, e_3, e_4$. Then the associated involution satisfies 
	\begin{equation} \label{Eqn: ExpressionForEllipticType}
		\langle -1 \rangle \cdot \deg^{\bbA^{1}}(i) = \langle \operatorname{Res}(\frac{\partial f}{\partial x_1}|S, \frac{\partial f}{\partial x_2}|S )  \rangle \text{ in $\operatorname{GW}(k)$.}
	\end{equation}
\end{pr}
\begin{rmk}
	
	Note that the resultant in \eqref{Eqn: ExpressionForEllipticType} should be understood as the resultant of homogeneous polynomials in $x_3$ and $x_4$. The choices of bases are not significant because different choices  would change the resultant by a perfect square, leaving the  class in $\operatorname{GW}(k)$ unchanged.
\end{rmk}
 \begin{rmk}
 For any line $\ell$, we may choose a basis such that $\ell$ corresponds to a subspace $S := L \cdot e_3 + L \cdot e_4$, where $L=k(\ell)$ is the field of definition of $\ell$. Proposition \ref{Prop: ExpressionForType} then implies the equality $\langle -1 \rangle \cdot \deg^{\bbA^{1}}(i) = \langle \operatorname{Res}(\frac{\partial f}{\partial x_1}|S, \frac{\partial f}{\partial x_2}|S )  \rangle$ in $\operatorname{GW}(L)$.
 \end{rmk}
\begin{proof}
	By Corollary~\ref{Corollary: Beauville}, it is enough to show that the right-hand side of \eqref{Eqn: ExpressionForEllipticType} equals the class of the discriminant of the fixed locus of $i$.  This fixed locus maps isomorphically onto the ramification locus of $\pi \colon \ell \to \bbP( T^{\vee})$, and we  compute by directly computing the discriminant of the ramification locus using Lemma~\ref{Lemma: ResidualPtsOnLine} as follows.

	 If we write $f = x_{1} P_{1} + x_{2} P_{2}$, then Lemma~\ref{Lemma: ResidualPtsOnLine} implies that the ramification locus is the locus where the polynomial 
	\begin{equation} \label{Eqn: ResidualCurve}
		a \cdot P_1(0, 0, x_{3}, x_{4})+ b \cdot P_2(0, 0, x_{3}, x_{4})
	\end{equation}
	in $x_{3}, x_{4}$  has a multiple root.  The ramification locus is thus the zero locus of $\operatorname{Disc}_{x_{3}, x_{4}}( a \cdot P_{1}(0, 0, x_{3}, x_{4}) + b \cdot P_{2}(0, 0, x_{3}, x_{4}))$, the discriminant of \eqref{Eqn: ResidualCurve} considered as a polynomial in $x_{3}$ and $x_{4}$.  Consequently the discriminant of the ramification locus is $\operatorname{Disc}_{a, b}( \operatorname{Disc}_{x_{3}, x_{4}}( a \cdot P_{1}(0, 0, x_{3}, x_{4}) + b \cdot P_{2}(0, 0, x_{3}, x_{4}))) \in k^{\ast}/ (k^{\ast})^{2}$.
	
	The right-hand of \eqref{Eqn: ExpressionForEllipticType} can also be described in terms of $P_{1}, P_{2}$.  Differentiating $f = x_{3} P_{1} + x_{4} P_{2}$, we get $\frac{\partial f}{\partial e_1}|S = P_{1}(0, 0, x_{3}, x_{4})$ and $\frac{\partial f}{\partial e_2} = P_{2}(0, 0, x_{3}, x_{4})$.  We now complete the proof by computing explicitly.   If  $P_1(0, 0, x_{3}, x_{4}) = \sum a_{i} x_{3}^{i} x_{4}^{2-i} \text{ and } P_{2}(0, 0, x_{3}, x_{4}) = \sum b_{i} x_{3}^{i} x_{4}^{2-i}$,  then resultant computations show
	\begin{align*}
		\operatorname{Res}(P_1(0, 0, x_{3}, x_{4}), P_2(0, 0, x_{3}, x_{4})) =& 	a_1^2 b_0 b_2-a_2 a_1 b_0 b_1-a_0 a_1 b_1 b_2 +a_2^2 b_0^2+a_0 a_2 b_1^2 \\&  +a_0^2 b_2^2-2 a_0   a_2 b_0 b_2 \\
								=& 1/16 \cdot \operatorname{Disc}_{a,b}(\operatorname{Disc}_{x_{3},x_{4}}( a \cdot P_{1}(0, 0, x_{3}, x_{4})+ b \cdot P_{2}(0, 0, x_{3}, x_{4}))).
	\end{align*}
\end{proof}

 \section{Euler number for relatively oriented vector bundles}\label{section:Euler_number_relatively_oriented_VB}
 
In this section, we define an Euler number in $\GW(k)$ for an algebraic vector bundle which is appropriately oriented and has a sufficiently connected space of global sections with isolated zeros. The definition is elementary in the sense that it can be calculated with linear algebra. Some duality theory from commutative algebra as in \cite{Beauville_residue} \cite{eisenbud77} \cite{scheja} is used to show the resulting element of $\GW(k)$ is well-defined, but no tools from $\bbA^1$-homotopy theory are needed. The precise hypothesis we use on sections is given in Definition \ref{df:ss'connected_isolated_zeros}. The precise hypothesis on orientations is that the vector bundle be relatively oriented as in Definition \ref{def:rel_orientation}. The vector bundle is assumed to be on a smooth $k$-scheme with $k$ a field.

An alternative approach using Chow-Witt groups or oriented Chow groups of Barge-Morel \cite{BargeMorel} and Fasel \cite{FaselGroupesCW} is developed in the work of Marc Levine \cite{Levine-EC} without the hypothesis on sections. Please see the introduction for further discussion.

Let $\pi: E \to X$ be a rank $r$ vector bundle on a smooth dimension $r$ scheme $X$ over $k$. In \cite[Definition 4.3]{morel}, an \textbf{orientation} of $E$ is a line bundle $L$ and an isomorphism $L^{\otimes 2} \cong \wedge^{\Top} E.$  Following Okonek--Teleman \cite{okonek14}, we make use of the related concept of a relative orientation. 

Let $\calT(X) \to X$ denotes the tangent bundle.

 \begin{df}\label{def:rel_orientation}
A \textbf{relative orientation} of $E$ is a pair $(L, j)$ consisting of a line bundle $L$ and an isomorphism $j: L^{\otimes 2} \stackrel{\cong}{\to} \Hom(\wedge^{\Top}\calT(X), \wedge^{\Top}E)$.
\end{df}

Assume furthermore that $\pi: E \to X$ is relatively oriented. On an open $U$ of $X$, we say that a section $s$ of $\Hom(\wedge^{\Top}\calT(X), \wedge^{\Top}E)$ is a \textbf{square} if its image under $$ \Gamma (U, \Hom(\wedge^{\Top}\calT(X), \wedge^{\Top}E) ) \cong \Gamma (U, L^{\otimes 2} )$$ is the tensor square of an element in $\Gamma (U, L)$. 

Let $p$ be a closed point of $X$, which as above is a smooth dimension $r$ scheme over $k$.

\begin{df}\label{df:Nisnevich_coordinates}
An \'etale map $$\phi: U \to \bbA_k^r = \Spec k[x_1, \ldots, x_r]$$ from an open neighborhood $U$ of $p$ to the affine space, which induces an isomorphism on the residue field of $p$ is called \textbf{Nisnevich coordinates} around $p$. \hidden{We want a Nisnevich cover so that the corresponding map on rings is surjective. We are using the coordinates on the target to make a presentation of the $\mathcal{O}_{Z,p}$}  
\end{df} 

\begin{lm}
There are Nisnevich coordinates around any closed point whose residue field is separable over $k$ for $r \geq 1$.
\end{lm}

(When $r=0$, this result does not hold. For a counter-example, consider $\Spec L \to \Spec k$ for a non-trivial, finite, separable extension $k \subseteq L$.)

\begin{proof}
Let $X \to \Spec k$ be smooth of dimension $r \geq 1$ and let $p$ be a closed point of $X$ such that $k \subseteq k(p)$ is separable. We may assume that $X$ is affine. Let $p$ also denote the ideal corresponding to $p$. By \cite[II Corollaire 5.10 and 5.9]{sga1}, there are $x_1, \ldots, x_r$ in $\mathcal{O}_X$ such that $x_1, \ldots, x_r$ generate $p$. Since $k \subseteq k(p)$ is separable, there is $x \in \mathcal{O}_X$ which generates $k(p)$ as an extension of $k$ by the primitive element theorem. If $dx$ is zero in $\Omega_{X/k}^1 \otimes k(p)$, then $d (x + x_1), d x_2 ,\ldots, d x_r$ is a $k(p)$-basis of $\Omega_{X/k}^1 \otimes k(p)$, whence $X \to \Spec k [x + x_1, x_2, \ldots, x_r]$ is \'etale at $p$. Furthermore, $x + x_1 = x$ modulo $p$. It follows that the map $X \to \Spec k [x + x_1, x_2, \ldots, x_r]$ gives Nisnevich coordinates around $p$. If $dx$ is nonzero in $\Omega_{X/k}^1 \otimes k(p)$, then we may use $dx$ as the first basis element in a basis formed from the spanning set $\{ dx, dx_1, \ldots, dx_r \}$. The map to $r$-dimensional affine space over $k$ corresponding to this basis gives Nisnevich coordinates.

%%%%%%%%%%%%
%%%%%%%%%%%%
\hidden{Let $X \to \Spec k$ be smooth of dimension $r \geq 1$ and let $p$ be a closed point of $X$ such that $k \subseteq k(p)$ is separable. We may assume that $X$ is affine. Let $p$ also denote the ideal corresponding to $p$. For any elements $x_1, \ldots, x_r$ of $\mathcal{O}_X$, the corresponding morphism $X \to \Spec k[x_1, \ldots, x_r]$ is \'etale at $p$ if and only if $dx_1, \ldots dx_r$ generate $\Omega_{X/k}^1$ at $p$ \cite[II Proposition 5.1]{sga1}, which is equivalent to $dx_1, \ldots dx_r$ forming a $k(p)$-basis of $\Omega_{X/k}^1 \otimes k(p)$ \hidden{By Nakayama's lemma generating $\Omega_{X/k}^1$ at $p$ is equivalent to generating modulo $p$. Furthermore, $\Omega_{X/k}^1 \otimes k(p)$ is $r$-dimensional}. By \cite[II Corollaire 5.10 and 5.9]{sga1}, there are $x_1, \ldots, x_r$ in $\mathcal{O}_X$ such that $x_1, \ldots, x_r$ generate $p$.  The corresponding morphism $X \to \Spec k[x_1, \ldots, x_r]$ is \'etale at $x$ by \cite[II Corollaire 5.8]{sga1}. Since $k \subseteq k(p)$ is separable, there is $x \in \mathcal{O}_X$ which generates $k(p)$ as an extension of $k$ by the primitive element theorem. 

If $dx$ is zero in $\Omega_{X/k}^1 \otimes k(p)$, then $d (x + x_1), d x_2 ,\ldots, d x_r$ is a $k(p)$-basis of $\Omega_{X/k}^1 \otimes k(p)$. Furthermore, $x + x_1 = x$ modulo $p$. Therefore, the map $X \to \Spec k [x + x_1, x_2, \ldots, x_r]$ gives Nisnevich coordinates around $p$.

If $dx$ is nonzero in $\Omega_{X/k}^1 \otimes k(p)$, then we may use $dx$ as the first basis element in a basis formed from the spanning set $\{ dx, dx_1, \ldots, dx_r \}$. The map to $r$-dimensional affine space over $k$ corresponding to this basis gives Nisnevich coordinates. }
\end{proof}

\hidden{\begin{example}
Let $k \subset L$ be a finite degree separable field extension. Then Nisnevich coordinates do not exist around the closed point of the smooth $k$-scheme $\Spec L \to \Spec k$.
\end{example}}

\begin{pr}\label{existence_Nis_coord}
There are Nisnevich coordinates around any closed point of a smooth $k$-scheme of dimension $r\geq1$.
\end{pr}

\begin{proof}
Nisnevich coordinates exist around closed points of smooth $r$-dimensional $k$-schemes when $k$ is infinite and $r\geq1$ by \cite[Chapter 8, Proposition 3.2.1]{knus}. Combining with the previous lemma, we have the claimed existence of Nisnevich coordinates. 
\end{proof}

We thank Alexey Ananyevskiy and Ivan Panin for the reference to \cite{knus}. 

Let $\phi$ be Nisnevich coordinates around $p$. Since $\phi$ is \'etale, the standard basis for the tangent space of $\bbA_k^r$ gives a trivialization for $TX \vert_U$. By potentially shrinking $U$, we may assume that the restriction of $E$ to $U$ is trivial.

\begin{df}\label{r_U-def}
 A trivialization of $E\vert_U$ will be called \textbf{compatible} with Nisnevich coordinates $\phi$ and the relative orientation if the element of $\Hom(\wedge^{\Top}\calT(X)\vert_U, \wedge^{\Top}E\vert_U)$ taking the distinguished basis of $\wedge^{\Top}\calT(X)\vert_U$ to the distinguished basis of $\wedge^{\Top}E\vert_U$ is a square.

Given $\phi$ and a compatible trivialization of $E\vert_U$, let $r_U$ in $\Gamma (U, L)$ denote an element such that $r_U^{\otimes 2}$ maps to the distinguished section of $\Hom(\wedge^{\Top}\calT(X)\vert_U, \wedge^{\Top}E\vert_U)$ under the relative orientation.
\end{df} 
 
Let $\sigma$ in $\Gamma(X, E)$ denote a section, and let $Z \subseteq X$ denote the closed subscheme $\{ \sigma=0\}$.

\begin{df}
A point $p$ of $Z$ is said to be an \textbf{isolated zero} of $\sigma$ if the local ring $\mathcal{O}_{Z,p}$ is a finite $k$-algebra. We say that the section $\sigma$ has \textbf{isolated zeros} if $\mathcal{O}_Z$ is a finite $k$-algebra.
\end{df}

A section $\sigma$ has isolated zeros if every zero is isolated and $X$ is connected.  

\begin{pr} The following are equivalent characterizations of an isolated zero of a section $\sigma$, and of $\sigma$ having isolated zeros.  
\begin{itemize}
\item $p$ is an isolated zero of $\sigma$ if and only if there is a Zariski open neighborhood $U$ of $p$ such that the set-theoretic intersection $U \cap Z$ is $p$, i.e., $U \cap Z= \{ p\}$.
\item $\sigma$ has isolated zeros if and only if $Z$ consists of finitely many closed points.
\end{itemize}
\end{pr}

\begin{proof}
If $p$ is an isolated zero of $\sigma$, then $\mathcal{O}_{Z,p}$ is dimension $0$. Since $p$ is a closed point, $\mathcal{O}_{Z,p}/p$ has dimension $0$. Let $Z^0$ be an irreducible component of $Z$ containing $p$. Since $Z^0$ is finite type over a field, $\dim Z^0 = \dim \mathcal{O}_{Z,p} + \dim \mathcal{O}_{Z,p}/p = 0$. Thus $Z^0$ is an irreducible dimension $0$ scheme which is finite type over $k$ and is therefore a single point, which must be $p$. Thus we may take $U$ to be the complement of the other irreducible components of $Z$.

If $p$ is a closed point with a Zariski open neighborhood $U$ such that  $U \cap Z= \{ p\}$, then $\dim Z^0 =0$. A zero dimensional finite type $k$-algebra is finite.

Since a zero dimensional Noetherian ring has finitely many points, if $\sigma$ has isolated zeros, then $Z$ has finitely many points. These points are closed because since $Z$ is closed, any specialization of a point of $Z$ is in $Z$, and since $Z$ is zero dimensional, there are no such specializations.

If $Z$ consists of finitely many closed points, then $Z$ is a zero dimensional finite type $k$-algebra, which is thus finite.
\end{proof}

We will use Nisnevich coordinates around $p$ and a local trivialization of $E$ to express the section $\sigma$ of $E$ as a function $\bbA^r_k \to \bbA^r_k$ plus an error term in a high power of the ideal corresponding to $p$. This will allow us to use a notion of local degree of a function $\bbA^r_k \to \bbA^r_k$ to define the local contribution to the Euler number. To do this, we will relate local rings $\mathcal{O}_{X,p}$ and $\mathcal{O}_{Z,p}$ of $p$ to the coordinate functions coming from the Nisnevich coordinates. The ideal corresponding to $p$ will also be denoted by $p$. 

\begin{lm}\label{lm:tangent_generation}
Let $p$ be an isolated zero of the section $\sigma$, and let $\phi: U \to \bbA_k^r = \Spec k[x_1, \ldots, x_r]$ be Nisnevich coordinates around $p$. Then \begin{itemize}
\item $\mathcal{O}_{Z,p}$ is generated as a $k$-algebra by $x_1,\ldots,x_r$. (We identify the $x_i$ with their pullbacks $\phi^* x_i$.)
\item For any positive integer $m$, the local ring $\mathcal{O}_{X,p}/p^{m}$ is generated as a $k$-algebra by $x_1,\ldots,x_r$. 
\end{itemize}
\end{lm}

\begin{proof}
Since $\mathcal{O}_{Z,p}$ is finite, there exists an $m$ such that $p^m = 0$ in $\mathcal{O}_{Z,p}$. Since $\mathcal{O}_{Z,p}$ is a quotient of $\mathcal{O}_{X,p}$, it thus suffices to show the second assertion. Let $q\subset k[x_1, \ldots, x_r]$ be the prime ideal $q = \phi(p)$. By construction of $\phi$, the induced map $k[x_1, \ldots, x_r]/q \to \mathcal{O}_{X,p}/p$ on residue fields is an isomorphism. We claim by induction that the map $k[x_1, \ldots, x_r] \to \mathcal{O}_{X,p}/p^{m}$ is a surjection. Given an element $y$ of $ \mathcal{O}_{X,p}/p^{m}$,  by induction on $m$, we can find an element $y'$ of the image such that $y - y'$ is in $p^{m-1}$. We can therefore express $y-y'$ as $y - y' = \sum_i a_i b_i$ where $a_i$ is in $p^{m-2}$ and $b_i \in p$. Since $\phi$ is \'etale, $\phi$ induces an isomorphism on cotangent spaces. It follows that there exist $a_i' \in p^{m-2}$ and $b_i' \in p$ in the image, such that $a_i - a_i'$ is in $p^{m-1}$ and $b_i - b_i'$ is in $p^2$. Then $\sum_i a_i b_i = \sum_i a_i' b_i'$ in  $ \mathcal{O}_{X,p}/p^{m}$ and the latter is in the image, showing the lemma.
\end{proof}

Let $p$ be an isolated zero of $\sigma$ as above. Choosing a compatible trivialization of $E\vert_U$, the section $\sigma$ becomes an $r$-tuple of functions $(f_1,\ldots, f_r)$, and each $f_i$ restricts to an element of the local ring $\mathcal{O}_{X,p}$. The local ring $\mathcal{O}_{Z,p}$ is the quotient $$ \mathcal{O}_{Z,p} \cong  \mathcal{O}_{X,p}/ \langle f_1,\ldots, f_r \rangle.$$ We furthermore have a commutative diagram $$\xymatrix{& \ar@{->>}[d] k[x_1,\ldots, x_r] \ar[dl]\\ \mathcal{O}_{X,p} \ar@{->>}[r] & \mathcal{O}_{Z,p}}.$$ Since $\mathcal{O}_{Z,p}$ is finite, there exists an $m \geq 1$ such that $p^m = 0$ in $\mathcal{O}_{Z,p}$. In particular, we have the equality of ideals $ \langle f_1,\ldots, f_r \rangle =  \langle f_1,\ldots, f_r \rangle + p^m$ in $\mathcal{O}_{X,p}$. By Lemma \ref{lm:tangent_generation}, $k[x_1,\ldots, x_r] $ surjects onto $\mathcal{O}_{X,p}/p^{2m}$. Therefore, we have $g_i$ in $k[x_1,\ldots, x_r] $ for $i=1,\ldots, r$ such that $g_i - f_i \in p^{2m} \subseteq p^{m+1}$. (We again identify the $g_i$'s with their images under $\phi^*$.) 

\begin{lm}\label{<gi>=<fi>Xp}
$\langle g_1, \ldots, g_r \rangle^e = \langle f_1,\ldots, f_r \rangle^e$ in $\mathcal{O}_{X,p}$ for all positive integers $e$.
\end{lm}

\begin{proof}
It suffices to show the lemma for $e=1$. Since $ \langle f_1,\ldots, f_r \rangle \supseteq  \langle f_1,\ldots, f_r \rangle + p^m \supseteq  \langle f_1,\ldots, f_r \rangle + p^{m+1}$, we have that $$\langle g_1, \ldots, g_r \rangle \subseteq \langle f_1,\ldots, f_r \rangle .$$ In the other direction, the $g_i$'s generate  $ \langle g_1,\ldots, g_r \rangle + p^m $ modulo $p^{m+1}$ because modulo $p^{m+1}$, the $g_i$'s are equal to the $f_i$'s. Thus $p^m \subseteq  \langle g_1,\ldots, g_r \rangle$, giving equality.
\end{proof}

Let $q\subset k[x_1, \ldots, x_r]$ be the prime ideal $q = \phi(p)$.

\begin{lm}\label{<gi>inkq=kqcap<gi>}
$\langle g_1, \ldots, g_r \rangle^e = (\phi^*)^{-1} (\langle f_1,\ldots, f_r \rangle^e)$ in $ k[x_1, \ldots, x_r]_q$ for all positive integers $e$.
\end{lm}

\begin{proof}
It suffices to show that the quotient map \begin{equation}\label{lm:<gi>inkq=kqcap<gi>:quotient_map}a: k[x_1, \ldots, x_r]_q/\langle g_1, \ldots, g_r \rangle^e \to k[x_1, \ldots, x_r]_q/ (\phi^*)^{-1} (\langle f_1,\ldots, f_r \rangle^e) \end{equation} is injective. Since $\phi$ is flat, the map of local rings $\phi^*: k[x_1, \ldots, x_r]_q \to \mathcal{O}_{X,p} $ is faithfully flat.  It thus suffices to show the injectivity of the map  $$a\otimes_{k[x_1, \ldots, x_r]_q} \mathcal{O}_{X,p}: \mathcal{O}_{X,p}/\langle g_1, \ldots, g_r \rangle^e \to \mathcal{O}_{X,p}/ \mathcal{O}_{X,p} ((\phi^*)^{-1} (\langle f_1,\ldots, f_r \rangle)^e).$$ This map is injective by Lemma \ref{<gi>=<fi>Xp}. 
\end{proof}

\begin{lm}\label{OZpfci}
$\mathcal{O}_{Z,p} \cong k[x_1, \ldots, x_r]_{q}/\langle g_1, \ldots, g_r \rangle$ is a finite complete intersection.  
\end{lm} 

\begin{proof}
By Lemma \ref{lm:tangent_generation}, the map $k[x_1, \ldots, x_r]_{q} \to \mathcal{O}_{Z,p}$ is surjective. The kernel is $(\phi^*)^{-1} \langle f_1,\ldots, f_r \rangle$, so the lemma follows by Lemma \ref{<gi>inkq=kqcap<gi>}.
\end{proof}

By \cite[Section 3]{scheja}, the presentation $k[x_1, \ldots, x_r]_{q}/\langle g_1, \ldots, g_r \rangle \cong \mathcal{O}_{Z,p}$ of the finite complete intersection $k$-algebra $\mathcal{O}_{Z,p}$ determines a canonical isomorphism \begin{equation}\label{SSThetaiso}\Hom_k (\mathcal{O}_{Z,p}, k) \cong \mathcal{O}_{Z,p}\end{equation} of $\mathcal{O}_{Z,p}$-modules.\hidden{The hypotheses to verify are at the beginning of Section 3. We have assumed that $\mathcal{O}_{Z,p}$ is finite over $k$. The nondegenerate hypothesis is defined on \cite[p. 175 (4)]{scheja}. By (1.1), for quotients of polynomial rings, the nondegenerate condition is equivalent to being flat, which is immediate for $A=k$. Since localization is exact, the condition  \cite[p. 180]{scheja} given on the sequence of elements of the kernel is more general then the condition that they be generators.} Let $\eta$ be the element of $\Hom_k (\mathcal{O}_{Z,p}, k)$ corresponding to $1$ in $\mathcal{O}_{Z,p}$ as in \cite[p. 182]{scheja}. 

\begin{lm}\label{lm:eta_indep_gi}
Let $m \geq 1$ be such that $p^m = 0$ in $\mathcal{O}_{Z,p}$. Choose $g_i$ in $k[x_1,\ldots, x_r] $ for $i=1,\ldots, r$ such that $g_i - f_i $ is in $p^{2m}$, and let $\eta$ be the corresponding element of $\Hom_k (\mathcal{O}_{Z,p}, k)$. Then $\eta$ is independent of the choice of $g_1, \ldots, g_r$.
\end{lm}

\begin{proof}
$\eta$ commutes with base-change, as can be seen from the construction \cite[Section 3]{scheja}, so we may assume that $k(p) = k$, and that $\phi(p)$ is the origin by translation. Let  $g_1',\ldots,g_r'$ be another choice of $g_1, \ldots, g_r$, and let $\eta'$ and $\eta$ denote the corresponding elements of $\Hom_k(\mathcal{O}_{Z,p},k)$. Since $\phi$ is a flat map of integral domains, $k[x_1,\ldots, x_r] \to \mathcal{O}_{X,p}$ is injective\hidden{For any nonzero $g$ in $k[x_1,\ldots, x_r]$, we have the injection given by multiplication by$g$. This must remain injective after taking the tensor product, so $g$ is not in the kernel.}. By construction, $g_i' - g_i$ is in $p^{2m} \cap k[x_1,\ldots, x_r]_q$. Since $p^m \subseteq \langle f_1, \ldots, f_r \rangle$, it follows that $p^{2m}\subseteq \langle f_1, \ldots, f_r \rangle^2$. Thus $p^{2m} \cap k[x_1,\ldots, x_r]_q \subseteq \langle f_1, \ldots, f_r \rangle^2 \cap k[x_1,\ldots, x_r]_q$. By Lemma \ref{<gi>inkq=kqcap<gi>}, it follows that $p^{2m} \cap k[x_1,\ldots, x_r]_q \subseteq \langle g_1, \ldots, g_r \rangle^2 $. Thus we may express $g_i' - g_i$ as a sum  $g_i' - g_i = \sum_{j=1}^r \tilde{c}_{i,j} g_j$ with $\tilde{c}_{i,j}$ in $\langle g_1, \ldots, g_r \rangle$. Let $c_{i,j} = \tilde{c}_{i,j}$ for $i \neq j$ and let $c_{i,i} = 1 + \tilde{c}_{i,i}$. Then $g_i' = \sum_{j=1}^r c_{i,j} g_j$. Let $c$ in $\mathcal{O}_{Z,p}$ denote the image of $\det \begin{pmatrix} c_{i, j} \end{pmatrix}$. By \cite[Satz 1.1]{SSResiduen}, for all $y$ in $\mathcal{O}_{Z,p}$ there is equality $\eta (y) = \eta' (c y)$. Since $\begin{pmatrix} c_{i, j} \end{pmatrix}$ is congruent to the identity modulo $\langle g_1, \ldots, g_r \rangle$, and $\langle g_1, \ldots, g_r \rangle$ is in the kernel of $k[x_1, \ldots, x_r] \to \mathcal{O}_{Z,p}$, it follows that $c=1$.
\end{proof}

The homomorphism $\eta$ defines a symmetric bilinear form $\beta$ on $\mathcal{O}_{Z,p}$ by the formula $$\beta(x,y) =\eta(xy),$$ which is furthermore nondegenerate because the map $ y \mapsto \eta(xy)$ in $\Hom_k (\mathcal{O}_{Z,p}, k) $ maps to $x$ in $\mathcal{O}_{Z,p}$ under the isomorphism \eqref{SSThetaiso}. 

Suppose that $\phi, \phi': U \to \Spec k[x_1, \ldots, x_r]$ are Nisnevich coordinates aound $p$ and $\psi, \psi': E \vert_U \to \mathcal{O}_U^r$ are local trivializations compatible with $\phi$ and $\phi'$, respectively. By Lemma \ref{lm:eta_indep_gi}, this data defines $\eta, \eta': \mathcal{O}_{Z,p} \to k$, respectively, and corresponding nondegenerate symmetric bilinear forms $\beta,\beta'$. Let $r_U$ and $r'_U$ denote elements of $\Gamma(U, L)$ as in Definition \ref{r_U-def} for $(\phi, \psi)$ and $(\phi', \psi')$ respectively. Note that $r_U$ and $r'_U$ are non-vanishing by construction, and therefore $r_U/r'_U$ defines an element of $\Gamma(U,\mathcal{O^*})$.

\begin{lm}\label{lm:eta=eta'(r/r')2}
$\beta$ is the pullback of $\beta'$ by the isomorphism $\mathcal{O}_{Z,p} \to \mathcal{O}_{Z,p}$ given by multiplication by $r_U/r'_U$.
\end{lm}

\begin{proof}
The lemma is equivalent to the assertion that $\eta (y) =  \eta' ((r_U/r'_U)^2 y)$ for all $y$ in $\mathcal{O}_{Z,p}$. 

Suppose first that $\phi' = \phi$ and that $r_U'=r_U$. Let $(f_1, \ldots, f_r)$ and $(f_1', \ldots, f_r')$ denote $\psi(\sigma)$ and $\psi'(\sigma)$ expressed as $r$-tuples of regular functions on $U$. Then there is $M: U \to \SL_r$ such that $$ f_j' = \sum_{i=1}^r M_{ji} f_i.$$ By Lemma \ref{lm:tangent_generation}, there exist $M'_{ij} \in k[x_1,\ldots, x_r]$ such that in $M'_{ij} - M_{ij}$ is in $p^{2m}$, where as before $m$ is chosen so that $p^m = 0$ in $\mathcal{O}_{Z,p}$. By potentially shrinking $U$, we may assume that $U$ is affine and $M'_{ij} - M_{ij} \subseteq p^{2m} \mathcal{O}(U)$. Define $g_j'$ in $k[x_1, \ldots, x_r]$ by $$g_j' = \sum_{i=1}^r M_{ji}' g_i.$$ By construction $f_j' - g_j'$ is in $p^{2m}$, and we may therefore use $g_1',\ldots, g_r'$ to compute $\eta'.$ Because $\det M = 1$, the difference $\det M' -1$ is in $p^{2m}$, which is therefore $0$ in $\mathcal{O}_{Z,p}$. Thus $\eta = \eta'$, as claimed.

Now suppose that $\phi' = \phi$, and $\psi' = A \psi$ where $A$ in $\GL_r \mathcal{O}_U$ is the matrix restricting to the identity on $\mathcal{O}_U^{r-1}$ and multiplying the last coordinate by $\alpha^2$ for $\alpha$ in $\Gamma(U, \mathcal{O}^*)$. As before, let  $(f_1, \ldots, f_r)$ and $(f_1', \ldots, f_r')$ denote $\psi(\sigma)$ and $\psi'(\sigma)$ expressed as $r$-tuples of regular functions on $U$, so $  ( f_1', \ldots, f_r') = (\alpha^2  f_1, \ldots, f_r)$. Then $ \eta' (\alpha^2 y) = \eta(y)$ for all  $y$ in $\mathcal{O}_{Z,p}$ by \cite[Satz 1.1]{SSResiduen}. Furthermore, the distinguished basis of $\wedge^r E\vert_U$ via $\psi$ is $\alpha^2$ times the distinguished basis via $\psi'$ by construction. Since $\phi=\phi'$, it follows that  $\alpha^2 (r'_U)^2 = r_U^2$, showing the claim.

Combining the previous two paragraphs, we see that the lemma holds when $\phi'=\phi$. 

Suppose given $\phi, \phi': U \to \Spec k[x_1, \ldots x_r]$ Nisnevich coordinates around $p$. By Lemma \ref{lm:tangent_generation}, $\phi$ and $\phi'$ induce surjections $\Spec k[x_1, \ldots x_r] \to \mathcal{O}_{X,p}/ p^{N}$ for any chosen positive integer $N$. We may therefore choose a map $\varphi: \Spec k[x_1, \ldots x_r] \to \Spec k[x_1, \ldots x_r]$ fitting into the commutative diagram $$\xymatrix{k[x_1, \ldots x_r]  \ar[rr]^{\varphi^*} \ar[dr]^{\phi^*} && \ar[dl]^{(\phi')^*} k[x_1, \ldots x_r]  \\ & \mathcal{O}_{X,p}/ p^{N}&}.$$ Let $ \tilde{\phi} = \varphi \circ \phi'$. It follows that $\tilde{\phi}: U \to \Spec k[x_1, \ldots x_r]$ determines Nisnevich coordinates around $p$, and $\tilde{\phi}^* x_i - \phi^* x_i $ is contained in $p^{N}$ for $i=1, \ldots r.$ The coordinates $\tilde{\phi}$  and $\phi$ determine trivializations $t_{\tilde{\phi}}, t_{\phi}: TX \vert_U \stackrel{\cong}{\longrightarrow} \mathcal{O}_U^r$. Let $A = t_{\tilde{\phi}} \circ t_{\phi}^{-1}$. By choosing $N$ sufficiently large, we may assume that $A \in \GL_r \mathcal{O}_{U}$ is congruent to the identity mod $p^{2m}$. Let $\psi: E\vert_U \to \mathcal{O}_U^r $ be a trivialization of $E \vert_U$ compatible with $\phi$. Define $\tilde{\psi}$ by $\tilde{\psi}=A \psi$. Then $\tilde{\psi}$ is a trivialization of $E \vert_U$ compatible with $\tilde{\phi}$ such that $\tilde{r}_U = r_U,$ where $\tilde{r}_U$ is defined by Definition \ref{r_U-def} for $(\tilde{\phi}, \tilde{\psi})$. Let  $(f_1, \ldots, f_r)$ and $(\tilde{f}_1, \ldots, \tilde{f}_r)$ denote $\psi(\sigma)$ and $\tilde{\psi}(\sigma)$ expressed as $r$-tuples of regular functions on $U$, so in particular $(\tilde{f}_1, \ldots, \tilde{f}_r) = A (f_1, \ldots, f_r)$. Choose $g_1, \ldots, g_r$ in $k[x_1, \ldots, x_r]$ such that $g_i - f_i$ is in $p^{2m}$. By Lemma \ref{lm:tangent_generation}, we may choose $\tilde{A} \in \GL_r k[x_1, \ldots,x_r]$ such that $\phi^*\tilde{A} - A$ is congruent to the identity mod $p^{2m}$. Define $(\tilde{g}_1, \ldots, \tilde{g}_r)$ by the matrix equation $ (\tilde{g}_1, \ldots, \tilde{g}_r) = \tilde{A}(g_1, \ldots, g_r)$. Then $\tilde{\phi'} \tilde{g}_i - \tilde{f}_i$ is in $p^{2m}$. By construction (Lemma \ref{OZpfci}), we obtain presentations $\mathcal{O}_{Z,p} \cong k[x_1, \ldots, x_r]_{q}/\langle g_1, \ldots, g_r \rangle$ and $\mathcal{O}_{Z,p} \cong k[x_1, \ldots, x_r]_{q}/\langle \tilde{g}_1, \ldots, \tilde{g}_r \rangle$ of $\mathcal{O}_{Z,p}$. Furthermore, mapping $x_i$ to $x_i$ for $i=1, \ldots, r$ determines a commutative diagram $$ \xymatrix{k[x_1, \ldots x_r]  \ar[rr]^{\varphi^*} \ar[dr]^{\phi^*} && \ar[dl]^{(\phi')^*} k[x_1, \ldots x_r]  \\ & \mathcal{O}_{Z,p}&}$$ because $p^m = 0$ in $\mathcal{O}_{Z,p}$. Let $\tilde{\eta}: \mathcal{O}_{Z,p} \to k$ denote the homomorphism corresponding to the presentation $\mathcal{O}_{Z,p} \cong k[x_1, \ldots, x_r]_{q}/\langle \tilde{g}_1, \ldots, \tilde{g}_r \rangle$ as in \cite[p. 182]{scheja}.  By \cite[Satz 1.1]{SSResiduen} there is equality $\eta = \tilde{\eta}$.

It follows that we may replace $\phi$ by $\tilde{\phi}$ and assume that $\phi = \varphi \circ \phi'$.  $\eta$ commutes with base-change, so we may assume that $k(p) = k$, and that $\phi(p)$ is the origin by translation. We may likewise assume that $\phi'(p)$ is the origin, and therefore that $\varphi$ takes the origin to the origin. Let $\psi': E\vert_U \to \mathcal{O}_U^r$ denote a trivialization of $E\vert_U$ compatible with $\phi'$. Since $\varphi$ is \'etale on a neighborhood of $0$, the Jacobian of $\varphi$ defines a map $T\varphi$ from this neighborhood to $\GL_r$. By possibly shrinking $U$, we therefore have a map $T\varphi \circ \phi': U \to \GL_r$. Using the canonical action of $\GL_r$ on the free sheaf of rank $r$, we obtain a second trivialization of $E\vert_U$ given by $\psi=(T\varphi \circ \phi') \psi': E\vert_U \to \mathcal{O}_U^r$, which is compatible with $\phi$, and $r_U = r_U'$. Let $(f_1', \ldots, f_r') \in \mathcal{O}_U^r$ denote the coordinate projections of $\sigma$ under the trivialization $\psi'$. It follows that the coordinate projections $(f_1, \ldots, f_r) \in \mathcal{O}_U^r$ of $\sigma$ under the trivialization $\psi$ are $T\varphi \circ \phi (f_1', \ldots, f_r')$. Choose $g_i$ in $k[x_1, \ldots, x_r]$ for $i=1, \ldots, r $ such that $(\phi)^*  g_i - f_i$ is in $p^{2m}$. Then we may define $g_i'$ by $(g_1', \ldots, g_r')=(T\varphi)^{-1}(\varphi^* g_1, \ldots,  \varphi^* g_r)$ and have that $(\phi')^*  g_i' - f'_i$ is in $p^{2m}$ as required. By \cite[Satz 1.5 and 1.1]{SSResiduen}, it follows that $\eta = \eta'$. We have thus reduced to the case where $\phi = \phi'$, completing the proof.

\end{proof}

\begin{df}\label{df:ind_ps}
The \textbf{local index} of $\sigma$ at $p$ is defined to be the element $\ind_p \sigma$ of $\GW(k)$ represented by the symmetric bilinear form $$ \beta (x,y) = \eta(xy),$$ where $x$ and $y$ are in $\mathcal{O}_{Z,p}$. 
\end{df}

\begin{co}
Suppose $r>0$. The local index $\ind_p \sigma$ exists at any isolated zero $p$ of $\sigma$, and $\ind_p \sigma$ is independent of the choice of \begin{itemize}
\item $\phi: U \to \bbA_{k}^r = \Spec k[x_1,\ldots, x_r]$
\item The chosen compatible trivialization of $E \vert_U$.
\item $g_1, \ldots, g_r$
\end{itemize}
\end{co}

\begin{proof}
By Remark \ref{existence_Nis_coord}, Nisnevich coordinates exist around $p$. By Lemma \ref{OZpfci} and the construction of Scheja-Storch (\cite[Section 3]{scheja}) discussed immediately after, $\ind_p \sigma$ exists. The independence of the choice of $g_1, \ldots, g_r$ follows from Lemma \ref{lm:eta_indep_gi}. The independence of the choice of Nisnevich coordinates and compatible trivialization of $E \vert_U$ follows from Lemma \ref{lm:eta=eta'(r/r')2}.
\end{proof}

The local index is moreover straightforward to compute. For example, if $p$ is a simple zero of $\sigma$ with a neighborhood isomorphic to the affine space $\Spec k[x_1,\ldots, x_r]$ and $k \subseteq k(p)$ is separable, then $$\ind_p \sigma = \langle \Tr_{k(p)/k}J \rangle,$$ where $\Tr_{k(p)/k}: k(p) \to k$ is the field trace and $J$ is the Jacobian determinant of $\sigma$, i.e., $J = \det \begin{pmatrix} \frac{\partial f_j}{ \partial x_i}\end{pmatrix}(p)$ where $\sigma$ is identified with the function $(f_1,\ldots,f_r): \bbA^r_k \to \bbA^r_k$ in a compatible local trivialization of $E$.

In general, one may reduce the computation of the local index to the case where $p$ is a $k$-point using descent. (In the case where the residue field extension $k \subseteq k(p)$ is separable, one can also base change to $k(p)$ and then apply the trace: See Proposition \ref{index=Trindk(p)}.) When $p$ is a $k$-point, one may then replace $\eta$ by any $k$-linear homomorphism $\eta_{\textrm{new}}: \mathcal{O}_{Z,p} \to k$ which takes the distinguished socle element to $1$. Under appropriate circumstances, for instance when $k$ is characteristic $0$, choosing such a homomorphism is equivalent to choosing a homomorphism which takes the Jacobian determinant $J(p) = \det( \frac{\partial g_i}{\partial x_j}(p))$ to $\dim_k \mathcal{O}_{Z,p}$. See \cite[p. 764]{eisenbud78} \cite{eisenbud77} \cite{khimshiashvili}. Then $\ind_p \sigma$ is represented by the bilinear form on $\mathcal{O}_{Z,p}$ taking $(x,y)$ to $\eta_{\textrm{new}}(xy)$. Below are some examples. The $\eta$ of Scheja-Storch is used here to show invariance in families below (Lemma \ref{family_e}).  

\begin{example}
The most fundamental example is where $X = \bbA_{k}^r$, $p= 0$, and $E = \mathcal{O}^r$, with $E$ given the canonical relative orientation. In this case, $\sigma$ can be viewed as a function $\sigma: \bbA^r_k \to \bbA^r_k$ and $\ind_p \sigma$ is the Grothendieck-Witt class of Eisenbud-Khimshiashvili-Levine, or equivalently the local $\bbA^1$-Brouwer degree of $\sigma$ as shown in \cite{KWA1degree}. Specifically, let $(f_1, \ldots, f_r)$ denote the coordinate projections of $\sigma$.  We may choose $a_{i,j} \in k[x_1, \ldots, x_r]$ such that $$ f_i = \sum_{j=1}^r a_{i,j} x_j.$$ The distinguished socle element is $\det \begin{pmatrix} a_{i, j} \end{pmatrix}$. Choose $\eta$ sending $\det \begin{pmatrix} a_{i, j} \end{pmatrix}$ to $1$. Then $\ind_p \sigma$ is represented by the bilinear form $\beta$ on $k[x_1, \ldots, x_r]_0/ \langle f_1, \ldots, f_r \rangle$ defined $\beta (x,y) = \eta(xy)$.
\end{example}

\begin{example}\label{ind:P1O(2n)}
Let $X = \bbP^1_k = \Proj k[x,y]$, and let $p$ be the point $p = [0,1]$. Consider the Cartier divisor $(2n)p$ and its associated line bundle $E = \mathcal{O}((2n)p)$ for $2n \geq 1$. Let $\sigma$ be the global section $\sigma=1$ in $\Gamma(\bbP^1_k, \mathcal{O}((2n)p))$. We specify a relative orientation $$\Hom (\calT(\bbP^1_k),  \mathcal{O}((2n)p)) \cong \mathcal{O}((n-1) p)^{\otimes 2} $$ as follows. On $U= \Spec k[x/y] \to \bbP^1_k$, the tangent bundle $\calT(\bbP^1_k)$ is trivialized by $\partial_{x/y}$ and $\mathcal{O}((2n)p))$ is trivialized by $(x/y)^{-2n}$. Similarly, on $W = \Spec k[y/x] \to \bbP^1_k$, the tangent bundle $\calT(\bbP^1_k)$ is trivialized by $\partial_{y/x}$ and $\mathcal{O}((2n)p))$ is trivialized by $1$. Thus $\Hom (\calT(\bbP^1_k),  \mathcal{O}((2n)p)) $ is trivialized on $U$ by $\alpha_U$, where $\alpha_U(\partial_{x/y}) =(x/y)^{-2n}$, and on $W$ by $\alpha_W$, where $\alpha_W(\partial_{y/x}) = 1$. On $\Spec k[x/y]  \cap \Spec k[y/x]$ we have the equality $- \frac{x^2}{y^2} \partial_{x/y} =  \partial_{y/x}$, whence $\alpha_U = -\frac{y^{2n-2}}{x^{2n-2}}\alpha_W$.  We give $E$ a relative orientation $$\Hom (\calT(\bbP^1_k),  \mathcal{O}((2n)p)) \stackrel{\cong}{\to} \mathcal{O}((n-1) p)^{\otimes 2} $$ defined by sending $\alpha_U$ to $-(\frac{y^{n-1}}{x^{n-1}})^{\otimes 2}$ and sending $\alpha_W$ to $1^{\otimes 2}$. We use this relative orientation to compute $\ind_p \sigma$ for all $n$. In fact, we'll use two different choices of Nisnevich coordinates around $0$ for $n=1$ and see directly that we compute the same local index, as we must.

First, use the Nisnevich coordinates around $0$ given by $\phi: U \to \Spec [x_1]$ where $\phi^* x_1 = x/y$. The trivialization $\psi: E \vert_U \to \mathcal{O}_U$ defined by $\psi((x/y)^{-2n}) = -1$ is compatible with $\phi$. Then $\sigma$ corresponds to $f_1 = - x_1^{2n}$, and we may define $g_1=f_1= - x_1^{2n}$. We obtain the presentation of $\mathcal{O}_{Z,p} $ given by $$ \mathcal{O}_{Z,p} \cong k[x_1]/ \langle -x_1^{2n} \rangle.$$ We may choose $\eta:  \mathcal{O}_{Z,p} \to k$ to be defined by $\eta ( - x_1^{2n-1} ) = 1$, $\eta (x_1^i) = 0$ for $i = 0,1, \ldots, 2n$. Then $$\ind_ p \sigma = n (\langle 1 \rangle + \langle -1 \rangle).$$

For comparison, assume that the characteristic of $k$ is not $3$ and use the Nisnevich coordinates around $0$ given by $\phi: U  - \{1\} \to \Spec [x_1]$ where $\phi^* x_1 = (x/y -1)^3.$ For computational simplicity, let $n=1$. Note that the distinguished basis element of $T_0 U$ determined by $\phi$ is then $\frac{1}{3 (x/y -1)^2} \partial_{x/y}$, from which it follows that the trivialization $\psi: E \vert_{U  - \{1\} } \to \mathcal{O}_U$ defined by $\psi(y^{2n}/x^{2n}) = -3 (x/y -1)^2 $ is compatible with $\phi$. Then $f_1 = \frac{- x^{2n} 3 (x/y -1)^2}{y^{2n}}$, and $$\mathcal{O}_{Z,p} \cong (k[x/y]/\langle  \frac{ x^{2n} } {y^{2n}}(x/y -1)^2\rangle)_p \cong k[x/y]/ \langle \frac{ x^{2n} } {y^{2n}} \rangle.$$ The integer $m$ from Lemma \ref{lm:tangent_generation} can be taken to be $m = 2n$. Using the assumption that $n=1$, we compute that $ -3  (\frac{1}{3}(x_1 + 1) )^2 \cong f_1 \mod \frac{x^4}{y^4}$\hidden{$\frac{1}{3}(x_1 + 1) = \frac{x}{y}(1 - x/y + \frac{1}{3}(x/y)^2)$. $f_1 = -3 (x/y)^2 + 6 (x/y)^3$. $(\frac{1}{3}(x_1 + 1) )^2 = (x/y)^2(1-2(x/y))$. }, whence the function $g_1$ can be taken to be $g_1 = \frac{-1}{3} (x_1 + 1)^2$. We obtain the presentation of  $\mathcal{O}_{Z,p} $ given by $$ \mathcal{O}_{Z,p} \cong k[x_1]/ \langle \frac{-1}{3} (x_1 + 1)^2 \rangle.$$ We may choose $\eta:  \mathcal{O}_{Z,p} \to k$ to be defined by $\eta (\frac{-1}{3}( x_1+1 )) = 1$ and  $\eta (1)= 0$. Then $$\ind_p \sigma = \langle 1 \rangle + \langle -1 \rangle,$$ and we recover the previous computation as desired.
\end{example}

For a separable field extension $k \subseteq L$, let $\Tr_{L/k}: \GW(L) \to \GW(k)$ denote the trace which takes a bilinear form $\beta: V \otimes V \to L$ over $L$ to the composition of $\beta$ with the field trace $L \to k$, now viewing $V$ as a vector space over $k$. This is sometimes called the Scharlau transfer, as in earlier versions of the present paper and \cite{Hoyois_lef}, although we caution the reader the term Scharlau transfer may also refer to different but related maps, as in \cite[Remark 1.16]{Fasel-Lectures_Chow-Witt} \cite[Theorem 4.1]{Scharlau-Quadratic_reciprocity_laws}. When $k \subseteq k(p)$ is separable, the trace reduces the computation of $\ind_p \sigma$ to the case where $p$ is rational. Namely, let $p$ be an isolated zero of $\sigma$ such that $k \subseteq k(p)$ is a separable extension. Let $X_{k(p)}$ denote the base change of $X$ to $k(p)$ and let $p_{k(p)}$ denote the canonical point of $X_{k(p)}$ determined by $p: \Spec k(p) \to X$. Let $\sigma_{k(p)}$ denote the base change of $\sigma$.

\begin{pr}\label{index=Trindk(p)}
$\ind_p \sigma = \Tr_{k(p)/k} \ind_{p_{k(p)}} \sigma_{k(p)}$.
\end{pr}

\begin{proof}
Let $k(p) \subseteq L$ be an extension such that $k \subseteq L$ is Galois. Let $\beta$ denote the bilinear form representing $\ind_p \sigma$ and $\beta_{k(p)}$ denote the bilinear form representing $ \ind_{p} \sigma_{k(p)}$ as in Definition \ref{df:ind_ps}. We identify the bilinear forms $L \otimes \beta$ and $L \otimes \Tr_{k(p)/k}\beta_{k(p)}$ together with the associated descent data.

The bilinear form $L \otimes \beta$ has underlying vector space $L \otimes \mathcal{O}_{Z, p}$. The bilinear form $\Tr_{k(p)/k} \beta_{k(p)}$ has underlying vector space given by $\mathcal{O}_{Z_{k(p)}, p}$, so $L \otimes \Tr_{k(p)/k} \beta_{k(p)}$ has underlying vector space given by $L \otimes_k \mathcal{O}_{Z_{k(p)}, p}$.

For each coset $g \Gal_{L/k(p)}$ of $\Gal_{L/k}$, there is a point $g p$ in $X (L)$, and a corresponding local ring $\mathcal{O}_{Z_{L}, g p}$. The ring $L \otimes \mathcal{O}_{Z,p}$ decomposes into idempotents corresponding to the $g p$, giving rise to an isomorphism \begin{equation}\label{ind=sumotherpoints} L \otimes \mathcal{O}_{Z,p} \cong \prod_{g \in \Gal_{L/k}/ \Gal_{L/k(p)}}   \mathcal{O}_{Z_{L}, g p}.\end{equation} It follows from the construction on \cite[p. 182]{scheja} that the restriction of the map $L \otimes \eta_{k,p} : L \otimes \mathcal{O}_{Z,p} \to L$ to an idempotent $\mathcal{O}_{Z_{L}, g p}$ is the corresponding $\eta_{L, g p}$.

The map $g: \mathcal{O}_{Z_{k(p)}, p} \to \mathcal{O}_{g Z_{k(p)}, g p}$ determines a quotient map $$L \otimes_k \mathcal{O}_{Z_{k(p)}, p}  \to L \otimes_{g k(p)}  \mathcal{O}_{Z_{g k(p)},g p}.$$ These maps determine an isomorphism $$ L \otimes_k \mathcal{O}_{Z_{k(p)}, p} \cong \prod_{g \in \Gal_{L/k}/ \Gal_{L/k(p)}} L \otimes_{g k(p)}  \mathcal{O}_{Z_{g k(p)}, g p}.$$ We have $L \otimes_{g k(p)}  \mathcal{O}_{Z_{g k(p)}, g p} \cong \mathcal{O}_{Z_{L}, g p}$. We therefore have constructed a $k$-linear isomorphism \begin{equation}\label{tr=sumotherpoints} L \otimes_k \mathcal{O}_{Z_{k(p)}, p} \cong \prod_{g \in \Gal_{L/k}/ \Gal_{L/k(p)}}   \mathcal{O}_{Z_{L}, g p}.\end{equation} By the functoriality of $\eta$, the pullback of $\eta_{L, g p}$ by $g :    \mathcal{O}_{Z_{L},  p} \to \mathcal{O}_{Z_{L}, g p}$ is $\eta_{L,p}$. By the proof of \cite[VII Theorem 6.1]{lam05}, it follows that the isomorphism \eqref{tr=sumotherpoints} takes $L \otimes \Tr_{k(p)/k} \beta_{k(p)}$ from the left hand side to the orthogonal direct sum $ \oplus_{g \in \Gal_{L/k}/ \Gal_{L/k(p)}} \beta_{L, g p}$. 

Combining with the previous (\eqref{ind=sumotherpoints} and \eqref{tr=sumotherpoints}) we have an isomorphism  $$ L \otimes \mathcal{O}_{Z,p} \cong L \otimes_k \mathcal{O}_{Z_{k(p)}, p}$$ taking $L \otimes \beta$ on the left to $L \otimes \Tr_{k(p)/k}\beta_{k(p)}$ on the right. 

There are canonical $\Gal(L/k)$ actions on $ L \otimes \mathcal{O}_{Z,p}$, $L \otimes_k \mathcal{O}_{Z_{k(p)}, p}$, and $ \prod_{g \in \Gal_{L/k}/ \Gal_{L/k(p)}}   \mathcal{O}_{Z_{L}, g p}$. Unwinding definitions shows that the isomorphisms \eqref{ind=sumotherpoints} and \eqref{tr=sumotherpoints} are equivariant, identifying the appropriate descent data. 

\end{proof}

\begin{df}\label{df:e(E,s)}
Let $\pi: E \to X$ be a rank $r$ relatively oriented vector bundle on a smooth dimension $r$ scheme $X$ over $k$, and let $\sigma$ be a section with isolated zeros. Define the \textbf{Euler number} $e(E, \sigma)$ of $E$ relative to $\sigma$ to be $e(E,\sigma) = \sum_{p \in Z} \ind_p \sigma$.
\end{df}

Let $\pi: E \to X$ be as in Definition \ref{df:e(E,s)}. Consider the pullback $\mathcal{E}$ of $E$ to $X \times \bbA_{k}^1$, and note that $\mathcal{E}$ inherits a relative orientation. For any closed point $t$ of $\bbA^1$, let  $\mathcal{E}_t$ denote the pullback of $\mathcal{E}$ to $X \otimes k(t)$. Similarly, given a section $s$ of $\mathcal{E}$, let $s_t$ denote the pullback of $s$.

\begin{lm}\label{family_e}
Let $\pi: E \to X$ be as in Definition \ref{df:e(E,s)}, and let $\mathcal{E}$ denote the pullback of $E$ to $X \times \bbA_{k}^1$. Suppose that $X$ is proper. Let $s$ be a section of $\mathcal{E}$ such that $s_t$ has isolated zeros for all closed points $t$ of $\bbA_{k}^1$. Then there is a finite $\mathcal{O}(\bbA_{k}^1)$-module equipped with a nondegenerate symmetric bilinear form $\beta$ such that for any closed point $t$ of $\bbA_{k}^1$, there is an equality $\beta_t = e(\mathcal{E}_t, s_t)$ in $k(t)$. 
\end{lm}

\begin{proof}
Let $L \to X$ and $L^{\otimes 2} \cong \Hom (\wedge^r \calT(X), \wedge^r E)$ be the relative orientation of $E$. Let $\mathcal{X} = X \times \bbA_{k}^1$, and let $\mathcal{L}$ be the pullback of $L$ by the projection $\mathcal{X} \to X$. The canonical isomorphism $\wedge^{r+1} \calT(X \times \bbA_{k}^1) \cong \wedge^r \calT(X)$ and the isomorphism $L^{\otimes 2} \cong \Hom (\wedge^r \calT(X), \wedge^r E)$ give rise to a canonical isomorphism $\mathcal{L}^{\otimes 2} \cong \Hom (\wedge^{r+1} \calT(\mathcal{X}), \wedge^r \mathcal{E}),$ which is the relative orientation of $\mathcal{E}$.

Let $\mathcal{Z} \hookrightarrow \mathcal{X}$ be the closed immersion defined by $s=0$. Since $X$ is proper over $\Spec k$, it follows that $\mathcal{X} \to \bbA_{k}^1$ is proper, whence $p: \mathcal{Z} \to  \bbA_{k}^1$ is proper. For any closed point $t$ of $\bbA_{k}^1$, the fiber $\mathcal{Z}_t \to \bbA_{k(t)}^1$ is the zero locus of the section $s_t$, which is finite by hypotheses. Thus $\mathcal{Z} \to  \bbA_{k}^1$ has finite fibers and is therefore finite because it was also seen to be proper. We construct $\beta$ on the finite $\mathcal{O}(\bbA_{k}^1)$-module $ p_* \mathcal{L}$. 

 Let $z$ be a closed point of $\mathcal{Z}$ and let $t$ be its image in $\bbA_{k}^1 = \Spec k[\tau]$. By Remark \ref{existence_Nis_coord}, we may choose Nisnevich coordinates around $z$ in $\mathcal{X}_t$. Therefore we have functions $\overline{x}_1, \ldots, \overline{x}_r$ in $\mathcal{O}_{\mathcal{X}_t}$ such that the $\overline{x}_i$ generate $k(z)$ over $k(t)$ and $d \overline{x}_1, \ldots, d\overline{x}_r$ are a basis of $\Omega_{\mathcal{X}_t/k(t)}$ at $z$. Let $x_i$ be an element of  $\mathcal{O}_{\mathcal{X}}$ lifting $\overline{x}_i$. It follows that $d x_1, \ldots, dx_r$ form a basis of the fiber of $\Omega_{\mathcal{X}/\bbA_{k}^1}$ at $z$. Furthermore, $\tau, x_1, \ldots, x_r$ generate $k(z)$ over $k$. Thus $\phi: U \to \Spec k[\tau, x_1, \ldots, x_r]$ define Nisnevich coordinates around $z$.
 
It follows from Lemma \ref{lm:tangent_generation} that $\mathcal{O}_{\mathcal{Z},z}$ is generated as a $k[\tau]_t$-algebra by $x_1, \ldots, x_r$. \hidden{pf1: Apply Lemma \ref{lm:tangent_generation} to show that $\tau, x_1, \ldots, x_r$ generate $\mathcal{O}_{\mathcal{Z},z}$ as a $k$-algebra   pf 2: We claim that $\mathcal{O}_{\mathcal{Z},z}$ is generated as a $k[\tau]_t$-algebra by $x_1, \ldots, x_r$. To see this, note that the cokernel $C$ of the homomorphism $k[\tau]_t[x_1, \ldots, x_r] \to \mathcal{O}_{\mathcal{Z},z}$ of $k[\tau]_t$-modules is a finite $k[\tau]_t$-module.  Furthermore, $C \subseteq t C$ by Lemma \ref{lm:tangent_generation}. Thus by Nakayama's Lemma, $C=0$, as claimed.}

Let $I \subset k[\tau]_t[x_1, \ldots, x_r]$ denote the kernel of the surjection $k[\tau]_t[x_1, \ldots, x_r] \to \mathcal{O}_{\mathcal{Z},z}$. Since the tensor product is right exact, the sequence $$k(t) \otimes_{k[\tau]_t} I \to k(t)[x_1, \ldots, x_r] \to \mathcal{O}_{\mathcal{Z}_t,z} $$ is exact. Since $k[\tau]_t \to k(t)$ is surjective, so is $I \to k(t) \otimes_{k[\tau]_t} I$. Thus we may choose $g_1, \ldots, g_r$ in $I$ lifting elements of $k(t)[x_1, \ldots, x_r] $ as in the construction of $\ind_{z} s_t$.

Let $q\subset k[\tau][x_1, \ldots, x_r]$ be the prime ideal determined by $ \phi(z)$. By Nakayama's Lemma and Lemma \ref{OZpfci}, $ g_1, \ldots, g_r$ generate the kernel of $k[\tau][x_1, \ldots, x_r]_{q} \to \mathcal{O}_{\mathcal{Z},z}.$ \hidden{pf: the cokernel $C'$ of $\langle g_1, \ldots, g_r \rangle \to I$ is a finitely generated $k[\tau ]_t[x_1, \ldots, x_r]_q$-module such that $k(t) \otimes_{k[\tau]_t[x_1, \ldots, x_r]_q} C' = 0$, whence $C'=0$ by Nakayama's Lemma.} Thus $$\mathcal{O}_{\mathcal{Z},z} \cong k[\tau][x_1, \ldots, x_r]_{q}/\langle g_1, \ldots, g_r \rangle,$$ expressing $k[\tau]_t \to \mathcal{O}_{\mathcal{Z},z}$ as a relative finite complete intersection. The morphism $k[\tau]_t \to \mathcal{O}_{\mathcal{Z},z}$ is furthermore flat by \cite[Lemma 10.98.3 Tag 00MD]{stacks-project}. Repeating this process for each $z$ in $p^{-1}(p(z))$, we have expressed $\mathcal{O}_{\bbb{A}^1,t} \to \mathcal{O}_{\mathcal{Z},z} \otimes_{\mathcal{O}_{\bbA^1}} \mathcal{O}_{\bbA^1,t} $ as a relative finite complete intersection. It follows that there is an open affine neighborhood $W$ of $t$ such that $p^{-1} W \to W$ is a flat relative finite complete intersection.\hidden{If $A \to B$ is a finite morphism with $A_t \to B otimes_A A_t$ expressed as $A_t[x_1, \ldots, x_r]/\langle g_1, \ldots, g_r \rangle$, then on an open subset of $\Spec B$, the elements $x_1, \ldots, x_r$ are module generators. Similarly, on an open subset, the kernel is generated as a module by $g_1, \ldots, g_r$. The map $\Spec B \to \Spec A$ is closed. The complement of the image of the complement of this open subset is an open subset of $\Spec A$ containing $t$, that may therefore be used for $W$. Note that the elements in this proof denoted by $x_i$'s or $g_i$'s are elements of a product ring whose components are the $x_i$'s and $g_i$'s of the paragraph.} 

Let $\eta_W$ denote the element in $\Hom_{\mathcal{O}_W}( \mathcal{O}_{p^{-1}W} ,\mathcal{O}_W )$ corresponding to $1$ under the canonical isomorphism \begin{equation*}\label{SSThetaisofamily}\Hom_{\mathcal{O}_W} (\mathcal{O}_{p^{-1}W}, {\mathcal{O}_W}) \cong \mathcal{O}_{p^{-1}W}\end{equation*} of $\mathcal{O}_{p^{-1}W}$-modules of \cite[Section 3]{scheja}. Let $\beta_{W,t}$ denote the nondegenerate bilinear form \begin{equation*}\beta_{W,t}(x,y)= \eta_{W,t}(xy)\end{equation*} $$\mathcal{O}_{\mathcal{Z}} (p^{-1}(W)) \otimes \mathcal{O}_{\mathcal{Z}} (p^{-1}(W))  \to \mathcal{O}_W,$$ specializing at $t$ to $e(\mathcal{E}_t, s_t)$.\hidden{Remember that for each point $t$, we have constructed a bilinear form so that the specialization is as claimed. Thus showing that this cover descends to a bilinear form also implies that the other specializations are as claimed.} 

Let $\mathcal{W}$ denote the set of those neighborhoods $W$. By Lemma \ref{lm:eta=eta'(r/r')2}, we may define elements $r_W$ in $p_* \mathcal{L} (W)$ for each $W$ in $\mathcal{W}$ such that for all $W$,$W'$ in $\mathcal{W}$, we have that $$\eta_W (y) = \eta_{W'} ((r_W/r_{W'})^2 y)$$ for all $y$ in $\mathcal{O}_{\mathcal{Z}} (p^{-1}(W \cap W'))$. The $r_W$ therefore define a descent datum on the $\beta_{W,t}$, which defines the bilinear form $\beta$ as claimed.

\end{proof}

Nondegenerate symmetric bilinear forms over $\bbA_k^1$ satisfy the property that their restrictions to any two $k$-rational points are stably isomorphic by a form of Harder's theorem (See \cite[Lemma 31]{KWA1degree}). Indeed, when $\chr k \neq 2$, such a bilinear form is pulled back from $\Spec k$ (\cite[VII Theorem 3.13]{Lam06}). This implies that Lemma \ref{family_e} shows that $e(\mathcal{E}_t,s_t)= e(\mathcal{E}_{t'},s_{t'})$, motivating the following definition. 

\begin{df}\label{df:ss'connected_isolated_zeros}
Say that two sections $\sigma$ and $\sigma'$ of $E$ with isolated zeros \textbf{can be connected by sections with isolated zeros} if there exist sections $s_i$ for $i=0,1,\ldots,N$ of $\mathcal{E}$ and rational points $t_{i}^-$ and $t_{i}^+$ of $\bbA_{k}^1$ for $i=1, 2, \ldots, N$ such that \begin{enumerate}
\item for $i = 0, \ldots, N$, and all closed points $t$ of $\bbA_{k}^1$, the section $(s_{i})_t$ of $E$ has isolated zeros.
\item $(s_0)_{t_{0}^-}$ is isomorphic to $\sigma$
\item $(s_N)_{t_{N}^+}$ is isomorphic to $\sigma'$
\item for $i = 0, \ldots, N-1$, we have that $(s_i)_{t_{i}^+}$ is isomorphic to $(s_{i+1})_{t_{i+1}^-}$. 
\end{enumerate}
\end{df}

\begin{co}\label{e(E)well-defined}
Let $\pi: E \to X$ be a rank $r$ relatively oriented vector bundle on a smooth, proper dimension $r$ scheme $X$ over $k$. \begin{enumerate}
\item The Euler numbers of $E$ with respect to sections $\sigma$ and $\sigma'$ with isolated zeros which after base change by an odd degree field extension $L$ of $k$ can be connected by sections with isolated zeros are equal: \begin{equation*}\label{co:eq:eEs=sEs'} e(E,\sigma) = e(E, \sigma').\end{equation*}
\item \label{connecting_generic_sections} Suppose there is a non-empty open subset $U$ of the affine space $\Gamma(X,E)$ of sections of $E$ such that any section in $U$ has isolated zeros, and such that any two sections in $U$ can be connected by sections with isolated zeros after base change by an odd degree field extension of $k$. Then the equality from the previous point holds for all sections $\sigma$ and $\sigma'$ in $U$. 
\end{enumerate}
\end{co}

\begin{proof}
It suffices to prove the first claim. For $k \subseteq L$ a field extension of finite odd dimension, tensoring with $L$ is an injective map $\GW(k) \to \GW(L)$. It follows that we may assume that $\sigma$ and $\sigma'$ can be connected by sections with isolated zeros (over $k$). Since $\sigma$ and $\sigma'$ can be connected by sections with isolated zeros, it suffices to show that for a section $s$ of $\mathcal{E}$ such that $s_t$ has isolated zeros for all closed points $t$, then $$e(E, s_t) = e(E, s_{t'}) $$ for $k$-rational points $t$ and $t'$ of $\bbA_{k}^1$.  By Lemma \ref{family_e}, there is a finite $\mathcal{O}(\bbA_{k}^1)$-module equipped with a nondegenerate symmetric bilinear form $\beta$ such that $\beta_t = e(E, s_t) $ and $\beta_{t'} = e(E, s_{t'})$. It therefore suffices to show that $\beta_t = \beta_{t'} $ in $\GW(k)$, which is true by the Serre problem for bilinear forms or Harder's theorem, for instance the version in \cite[Lemma 31]{KWA1degree}. 
\end{proof}

Note that if $U$ and $U'$ are two open sets of $\Gamma(X,E)$ satisfying the hypotheses of condition \eqref{connecting_generic_sections} in Corollary~\ref{e(E)well-defined}, then $U \cap U'$ is nonempty and $e(E,\sigma) = e(E, \sigma')$ for any $\sigma$ and $\sigma'$ in $(U \cup U')(k)$.

\begin{df}\label{df:e(E)well-defined}
If condition~\eqref{connecting_generic_sections} in Corollary \ref{e(E)well-defined} is satisfied, define the {\em Euler number $e(E)$ of $E$} by $e(E)= e(E,\sigma)$ for any section $\sigma$ in $U$.
\end{df}

\begin{example}
Let $X=\bbP_{k}^1$ and $E= \mathcal{O}(2n)$, with $E$ oriented as in Example \ref{ind:P1O(2n)}. Then $e(\mathcal{O}(2n)) = n(\langle 1 \rangle + \langle -1 \rangle)$ by Example \ref{ind:P1O(2n)}.
\end{example}

\section{Counting lines on the cubic surface} \label{Section: Counting}
We now apply the results from the previous sections to count the lines on a smooth cubic surface, i.e.~to prove Theorem~\ref{Theorem: MainTheorem}.  Recall our approach is to identify the arithmetic count of lines (the expression in \eqref{Eqn: FinalLineCount}) with the Euler number of a vector bundle on the Grassmannian $G := \operatorname{Gr}(4, 2)$ of lines in projective space $\bbP_{k}^3$.  As before, we let $x_{1}, x_{2}, x_{3}, x_{4} \in (k^{\oplus 4})^{\vee}$ denote the basis dual to the standard basis for $k^{\oplus 4}$. Given any basis $\underline{e} = \{ e_1, e_2, e_3, e_4\}$ of $k^{\oplus 4}$, let $\phi_1, \phi_2, \phi_3, \phi_4$ denote the dual basis. 

We begin by orienting the relevant vector bundle.
\begin{df}
	Let $\mathcal{S}$ and $\mathcal{Q}$  respectively denote the tautological subbundle and quotient bundle on $\operatorname{Gr}(4, 2)$.  Set
	\[
		\mathcal{E} := \operatorname{Sym}^{3}( \mathcal{S}^{\vee}).
	\]
	Given a degree $3$ homogeneous polynomial $f \in \operatorname{Sym}^{3}( (k^{\oplus 4})^{\vee})$,  we define the global section $\sigma_{f}$ to be the image of $f$ under the homomorphism $\operatorname{Sym}^{3}( (\calO)^{\oplus 4}) \to \calE$ induced by the inclusion $\calS \subset \calO^{\oplus 4}$.
\end{df}
Intuitively, the fiber of $\calE$ at a $2$-dimensional subspace $S \subset k^{\oplus 4}$ is the space $\operatorname{Sym}^{3}(S^{\vee})$ of homogeneous degree $3$ polynomials on $S$, and the image of  $\sigma_{f}$ in the fiber is the restriction $f|S$.

The tangent bundle to $G$ admits a natural description in terms of tautological bundles:
\begin{align*}
	\calT(G) 	=& \ShHom(\mathcal{S}, \mathcal{Q}) \\
			=& \calS^{\vee} \otimes \calQ.
\end{align*}

We now exhibit an explicit relative orientation of $\calE$ using the standard open cover constructed using the following definition and lemma.

\begin{df} \label{Def: Bases}
	If $\underline{e} = \{ e_1, e_2, e_3, e_4 \}$ is a basis for $k^{\oplus 4}$, then define the following elements of $( k[x, x', y, y'])^{\oplus 4}$
	\begin{gather*}
		\widetilde{e}_{1} := e_1, \widetilde{e}_{2} := e_2, \\
		\widetilde{e}_{3} := x e_1 + y e_2 + e_3 \text{ and } \widetilde{e}_{4} :=  x' e_1 + y' e_2 + e_4.
	\end{gather*}
	This is a basis, and we define $\widetilde{\phi}_1, \widetilde{\phi}_2, \widetilde{\phi}_3, \widetilde{\phi}_4$ to be the dual basis.
\end{df}

\begin{lm}	 \label{Lemma: GrassmannianChart}
	If $\underline{e} = \{ e_1, e_2, e_3, e_4 \}$ is a basis for $k^{\oplus 4}$, then the morphism 
	\begin{equation}  \label{Eqn: ChartOfG}
		\Spec( k[x, x', y, y'] ) = \bbA^{4}_{k} \to G
	\end{equation}
	with the property that $\calS$ pulls back to the subspace 
	\begin{equation} \label{Eqn: GrassmannianChart}
		k[x, x', y, y'] \cdot \widetilde{e}_{3} + k[x, x', y, y'] \cdot \widetilde{e}_{4} \subset ( k[x, x', y, y'])^{\oplus 4}.
	\end{equation}
	is an open immersion.
\end{lm}
\begin{proof}
	Given a surjection $q \colon \calO^{\oplus 4}_{G} \to \calQ$, we can form the composite $\calO^{\oplus 2}_{G} \to \calO^{\oplus 4}_{G} \stackrel{q}{\longrightarrow} \calQ$ with the homomorphism  $\mathcal{O}^{\oplus 2}_{G} \to \mathcal{O}^{\oplus 4}_{G}$ defined by $(a, b) \mapsto a e_1 + b e_2$. The morphism $\Spec( k[x, x', y, y'] ) = \bbA^{4}_{k} \to G$ represents the subfunctor of $G$ parameterizing surjections $q$ such that this composition is an isomorphism.  The subfunctor is open by \cite[9.7.4.4]{EGA1}.
\end{proof}
The morphism \eqref{Eqn: ChartOfG} alternatively can be described  in terms of projective geometry.  If $H \subset \bbP^{3}_{k}$ is the hyperplane that corresponds to the span of $e_1, e_2, e_4$ and $H'$ to the span of $e_1, e_2, e_3$, then \eqref{Eqn: ChartOfG} is the (restriction of) the rational map $H \times_{k} H' \dashrightarrow G$ that sends a pair of points to the line they span. Indeed, when $\underline{e}$ is the standard basis, the line corresponding to the subspace in \eqref{Eqn: GrassmannianChart} is the line parameterized by $[S, T] \mapsto [x S + x' T, y S + y' T, S, T]$.  This is a general line that meets the hyperplanes $\{ x_3 =0\}$ and $\{ x_4 =0\}$.

\begin{df}
	Given a basis $\underline{e}$ for $k^{\oplus 4}$, define $U(\underline{e}) \subset G$ to be the image of \eqref{Eqn: ChartOfG}.
\end{df}
The collection $\{ U(\underline{e}) \}$ is the desired standard affine open cover.  Over $U(\underline{e})$,  $\calS^{\vee}$ is trivialized by the basis $\{ \tilde{\phi}_{3}, \tilde{\phi}_{4} \}$. Let $\overline{\tilde{e}_{1}}$ denote the image of $\tilde{e}_{1}$ in $Q(U)$. Then $Q$ is trivialized by $\{ \overline{\tilde{e}_{1}}, \overline{\tilde{e}_{2}}\}$ over this same open set.

\begin{pr}\label{Pr:BasisForTangent} 
As $\{e_1, e_2, e_3, e_4 \}$ varies among all basis of $k^4$, the local trivializations given by the sections $\{ \tilde{\phi}_{3} \otimes \overline{\tilde{e}_{1}}, \tilde{\phi}_{3} \otimes \overline{\tilde{e}_{2}} , \tilde{\phi}_{4} \otimes \overline{\tilde{e}_{1}}, \tilde{\phi}_{4} \otimes \overline{\tilde{e}_{2}}  \}$ are compatible with an orientation of $T\Gr(4,2)$.
\end{pr}

%It suffices to prove the following lemma.

We make explicit the change of basis. Let $\{b_1, b_2, b_3, b_4 \}$ be a basis of $k^4$ and use the local coordinates $$V =  U(\{b_1, b_2, b_3, b_4 \}) = \Spec k[w,w',z,z'] \to \Gr(4,2)$$ of $\Gr(4,2)$ described above \eqref{Eqn: ChartOfG}, so $(w,w',z,z')$ corresponds to the span of  $\{\tilde{b}_{3}, \tilde{b}_{4} \}$, where $\{\tilde{b}_{1}, \ldots, \tilde{b}_{4} \}$ is the basis of $k^{4}$ defined by $$\tilde{b}_i =  \begin{cases} b_i  &\mbox{for } i = 1, 2 \\ 
w b_1 + z b_2 + b_3 & \mbox{for } i=3
\\ 
w' b_1 + z' b_2 + b_4 & \mbox{for } i={4}. \end{cases} $$ Let $\{\tilde{\theta}_{1}, \ldots, \tilde{\theta}_{4} \}$ and $\{\theta_1, \ldots, \theta_4 \}$ denote the dual bases of $\{\tilde{b}_{1}, \ldots, \tilde{b}_{4} \}$ and $\{b_1, b_2, b_3, b_4 \}$ respectively. 

On $V \cap U$, the trivializations of $T\Gr(4,2)$ corresponding to the bases $$\{ \tilde{\theta}_{3} \otimes \overline{\tilde{b}_{1}}, \tilde{\theta}_{3} \otimes \overline{\tilde{b}_{2}} , \tilde{\theta}_{4} \otimes \overline{\tilde{b}_{1}}, \tilde{\theta}_{4} \otimes \overline{\tilde{b}_{2}}  \}$$ and $$\{ \tilde{\phi}_{3} \otimes \overline{\tilde{e}_{1}}, \tilde{\phi}_{3} \otimes \overline{\tilde{e}_{2}} , \tilde{\phi}_{4} \otimes \overline{\tilde{e}_{1}}, \tilde{\phi}_{4} \otimes \overline{\tilde{e}_{2}}  \}$$ give rise to clutching functions $M_{be}=M_{eb}^{-1}$ in $\GL_4 V \cap U$.

\begin{lm}
$\det M_{be}$ is a square of an element of $\calO(U \cap V)$.
\end{lm}

\begin{proof}
For any $i$ and $j$ in $\{1,2,3,4\}$, the sections $\tilde{\phi}_{i}$ and $\tilde{\theta}_{i}$ are in $(\calO^4)^{\vee}(U \cap V)$, and the sections $\tilde{b}_{j}$ and $\tilde{e}_{j}$ are in $\calO^4(U \cap V)$. The expressions $\tilde{\phi}_{i}(\tilde{b}_{j})$ and $\tilde{\theta}_{i}(\tilde{e}_{j})$ thus determine regular functions, i.e. elements of  $\calO(U \cap V)$.

The change of basis matrix relating the bases $\{ \tilde{\phi}_3,\tilde{\phi}_4\} $ and $\{ \tilde{\theta}_3,\tilde{\theta}_4\} $ of $\calS^{\vee}(U \cap V)$ is $$A = \begin{bmatrix}
\tilde{\theta}_3 (\tilde{e}_3)& \tilde{\theta}_4 (\tilde{e}_3) \\
\tilde{\theta}_3 (\tilde{e}_4)& \tilde{\theta}_4 (\tilde{e}_4)
\end{bmatrix} $$

Note that $\det A$ is a regular function. Similarly, we have a $2$ by $2$ matrix $B$ relating the bases $\{ \overline{\tilde{b}_{1}}, \overline{\tilde{b}_{2}}\}$ and $\{ \overline{\tilde{e}_{1}}, \overline{\tilde{e}_{2}}\}$ of $\calQ(U \cap V)$ $$B = \begin{bmatrix}
\tilde{\phi}_1 (\tilde{b}_1)&\tilde{\phi}_1 (\tilde{b}_2)\\
 \tilde{\phi}_2 (\tilde{b}_1)& \tilde{\phi}_2 (\tilde{b}_2)
\end{bmatrix} $$ and $\det B$ is a regular function. 

By definition, $M_{be}$ is the change of basis matrix relating $\{ \tilde{\theta}_3,\tilde{\theta}_4\} \otimes \{ \overline{\tilde{b}_{1}}, \overline{\tilde{b}_{2}}\}$ and $\{ \tilde{\phi}_3,\tilde{\phi}_4\} \otimes  \{ \overline{\tilde{e}_{1}}, \overline{\tilde{e}_{2}}\}$ . Therefore $M_{be}$ is the tensor product $M_{be} = A \otimes B$, and $\det M_{be} = (\det A)^2 (\det B)^2$ because $A$ and $B$ are both $2$ by $2$ matrices. It follows that $\det M_{be}$ is the square of a regular function as desired.
\end{proof}

\begin{proof}
(of Proposition \ref{Pr:BasisForTangent}) The cocycle associating $U \cap V$ to $\det A \det B$ in $\calO^*(U \cap V)$ determines a line bundle $\calL$ with distinguished triaivalizations on the open cover $\{U(\underline{e})  \}$. ( Here $U=U(\underline{e}) $ and $V=U(\{b_1,b_2,b_3,b_4 \})$ as above.) Under these trivializations, sending the wedge product $ \tilde{\phi}_{3} \otimes \overline{\tilde{e}_{1}} \wedge \tilde{\phi}_{3} \otimes \overline{\tilde{e}_{2}} \wedge \tilde{\phi}_{4} \otimes \overline{\tilde{e}_{1}} \wedge \tilde{\phi}_{4} \otimes \overline{\tilde{e}_{2}} $ to $1 \otimes 1$ determines the desired isomorphism between $\det T\Gr(4,2)$ and $\calL^{\otimes 2}$.
\end{proof}

One shows similarly that:

\begin{pr}\label{Pr:BasisForSym}
As $\{e_1, e_2, e_3, e_4 \}$ varies among all basis of $k^4$, the local trivializations given by the sections $\{ \widetilde{\phi}_{3}^{3}, \widetilde{\phi}_{3}^{2} \widetilde{\phi}_{4}, \widetilde{\phi}_{3}  \widetilde{\phi}_{4}^{2}, \widetilde{\phi}_{4}^{3}   \}$ are compatible with an orientation of $\calE$.
\end{pr}

Let  $v(\underline{e})$ denote the section of $\ShHom(\wedge^{4} \calT(G), \wedge^{4} \calE)|U(\underline{e})$ that maps the wedge product of the sections in Proposition \ref{Pr:BasisForTangent} to the wedge product of the sections in Proposition \ref{Pr:BasisForSym}.
%%%%%%%%%%%%

\begin{co} \label{Lemma: ConstructOrientation}
	There is a unique orientation (up to isomorphism) of  $\ShHom( \wedge^{4} \calT(G), \wedge^{4} \calE) $ such that $v(\underline{e})$ is a square for all $\underline{e}$.
\end{co}

\begin{rmk}
Let $\calL=\wedge^{2} \calS^{\vee}$. It can be shown that the map \begin{equation} \label{Eqn: OrientationIso}
		 j \colon \ShHom( \wedge^{4} \calT(G), \wedge^{4} \calE) \to \calL^{\otimes 2}
	\end{equation}
	such that the restriction to $U(\underline{e})$ sends $v(\underline{e})$ to $ (\widetilde{\phi}_3 \wedge \widetilde{\phi}_4) \otimes (\widetilde{\phi}_3 \wedge \widetilde{\phi}_4)$ for all bases $\underline{e}$ is an isomorphism defining the distinguished relative orientation of Lemma \ref{Lemma: ConstructOrientation}. 
	
	A remark about this choice of orientation.  The line bundle $\calL$ is the unique square root of $\ShHom( \wedge^{4} \calT(G), \wedge^{4}  \calE)$, as $\operatorname{Pic}(G)$ is torsion-free, so there is no other possible choice of line bundle. There are other choices of isomorphism $\ShHom( \wedge^{4} \calT(G), \wedge^{4} \calE) \cong \calL^{\otimes 2}$, namely the isomorphisms $a \cdot j$ for $a \in k^{\ast}$.  The isomorphism $j$ is distinguished by the property that  the local index of $\sigma_{f}$ at a zero equals the type of the corresponding line (as defined in Section~\ref{Section: ClassificationOfLines}), i.e.~$j$  makes Corollary~\ref{Corollary: LocalIndexIsTypeRestated} hold.  The isomorphism $j$ also has the property that it is defined over $\bbZ$, and these two properties uniquely characterize $j$.
\end{rmk}

Having defined a relative orientation of $\calE$, we now identify the local index of $\sigma_{f}$ at a zero with the type of line.

\begin{lm} \label{Lemma: ExplicitDifferential}
Let $f \in k[x_1, x_2, x_3, x_4]$ be a cubic homogeneous polynomial.  If $S \subset k^{\oplus 4}$ has the property that $f|S=0$, then the differential of $\sigma_{f}$ at the corresponding $k$-point of $G$ is the map 
	\[
		S^{\vee} \otimes Q \to \operatorname{Sym}^{3}(S^{\vee})
	\]
	defined by
	\begin{equation} \label{Eqn: FormulaForDifferential}
		\phi \otimes (v+S) \mapsto (\frac{\partial f}{\partial v})|S \cdot \phi.
	\end{equation}
	
	Here $\frac{\partial f}{\partial v}$ is the directional derivative of $f$ in the direction of $v$.  (The derivative depends on $v$, but its restriction to $S$ depends only on the coset $v + S$.)
\end{lm}
\begin{proof}
	When $k=\bbR$, this is \cite[Lemma~26]{okonek14}.  Rather than adapting that proof to the present setting, we prove the lemma by computing everything in terms of a trivialization.  Given $S$, pick a standard open neighborhood $U(\underline{e})$ associated to some basis such that $S = k \cdot e_{3} + k \cdot e_{4}$ and then write $f = \sum a_{\underline{i}} \phi_{1}^{i_1} \phi_{2}^{i_2} \phi_{3}^{i_3} \phi_{4}^{i_4}$.  Computing partial derivatives, we see that the map defined by \eqref{Eqn: FormulaForDifferential} is characterized by
	\begin{gather}
		\phi_{3} \otimes e_{1} \mapsto  (a_{1,0,2,0} \phi_{3}^2  + a_{1,0,1,1} \phi_{3} \phi_{4} + a_{1,0,0,2} \phi_{4}^2) \cdot \phi_{3} 	\label{Eqn: ExplicitFormulaForDifferential} \\
		\phi_{4} \otimes e_{1}	\mapsto (a_{1,0,2,0} \phi_{3}^2  + a_{1,0,1,1} \phi_{3} \phi_{4} + a_{1,0,0,2} \phi_{4}^2) \cdot \phi_{4} 	\notag \\
		\phi_{3} \otimes e_{2}	\mapsto  (a_{0,1,2,0} \phi_{3}^2 + a_{0,1,1,1} \phi_{3} \phi_{4} + a_{0,1,0,2} \phi_{4}^2) \cdot \phi_{3} 	\notag \\
		\phi_{4} \otimes e_{2} \mapsto (a_{0,1,2,0} \phi_{3}^2 + a_{0,1,1,1} \phi_{3} \phi_{4} + a_{0,1,0,2} \phi_{4}^2) \cdot \phi_{4}. \notag
	\end{gather}
	
	We compare this function to the derivative of $\sigma_{f}$ by computing as follows.  Trivializing the restriction of $\calE$ using the sections $\widetilde{\phi}_{3}^3, \widetilde{\phi}_{3}^{2} \widetilde{\phi}_{4}, \widetilde{\phi}_{3} \widetilde{\phi}_{4}^{2}, \widetilde{\phi}_{4}^{3}$, the section $\sigma_{f}$ gets identified with the function $\Spec(k[x, x', y, y'])= \bbA^{4}_{k} \to \bbA^{4}_{k}$ whose components are the coefficients of $\widetilde{\phi}_{3}^3, \widetilde{\phi}_{3}^{2} \widetilde{\phi}_{4}, \widetilde{\phi}_{3} \widetilde{\phi}_{4}^{2}, \widetilde{\phi}_{4}^{3}$ in 
	\begin{align*}
		f|S 	=&  \sum a_{\underline{i}} \phi_{1}^{i_1} \phi_{2}^{i_2} \phi_{3}^{i_{3}} \phi_{4}^{i_4} \\
			=& \sum a_{\underline{i}}  (x \widetilde{\phi}_{3} + x' \widetilde{\phi}_{4})^{i_{1}}  (y \widetilde{\phi}_{3} + y' \widetilde{\phi}_{4})^{i_{2}} \widetilde{\phi}_{3}^{i_3} \widetilde{\phi}_{4}^{i_4}.
	\end{align*}
	The partial derivative of this function with respect to $x$ at $(x, x', y, y') = (0, 0, 0, 0)$ is 
	\begin{align*}
		\frac{\partial f|S}{\partial x}(0) 	=&	\sum a_{\underline{i}}  i_{1} (0 \cdot \widetilde{\phi}_{3} + 0 \cdot \widetilde{\phi}_{4})^{i_{1}-1} \cdot \widetilde{\phi}_{3} \cdot  (0 \cdot \widetilde{\phi}_{3} + 0  \cdot \widetilde{\phi}_{4})^{i_{2}} \widetilde{\phi}_{3}^{i_3} \widetilde{\phi}_{4}^{i_4}  \\
								=&	a_{1, 0, 2, 0} \widetilde{\phi}_{3}^{3} + a_{1, 0, 1, 1} \widetilde{\phi}^{2}_{3} \widetilde{\phi}_{4} +a_{1, 0, 0, 2} \widetilde{\phi}_{3} \widetilde{\phi}_{4}^{2}, 	\end{align*}
	which is the image of $\phi_3 \otimes e_1$ under \eqref{Eqn: ExplicitFormulaForDifferential} and similarly with the other derivatives.

\end{proof}

We now relate the Euler number of $\calE$ to the lines on a smooth cubic surface.  Observe that, by construction, the zero locus of $\sigma_{f}$ is the set of lines contained in the cubic surface $\{ f=0 \}$.

\begin{lm} \label{Lemma: LocalIndexIsType}
Let $f \in k[x_1, x_2, x_3, x_4]$ be a cubic homogeneous polynomial.  Then the derivative of   $\sigma_{f}$ at a zero defined by a subspace  $S = k \cdot e_{3} + k \cdot e_{4} \subset k^{\oplus 4}$ equals
	\[
		 \operatorname{Res}( \frac{\partial f}{\partial e_1}(x e_{3}+y e_{4}), \frac{\partial f}{e_2}(x e_{3}+y e_{4})) \text{ in $k/(k^{\ast})^{2}$}
	\]
	for $e_1, e_2 \in k^{\oplus 4}$ such that $e_{1}, e_{2}, e_{3}, e_{4}$ forms a basis.
\end{lm}
\begin{proof}
	The matrix of the differential $\Hom( S, Q) \to \operatorname{Sym}^{3}(S^{\vee})$ with respect to the bases $\phi_{3} \otimes e_{1}, \dots, \phi_{4} \otimes e_2$ and $\phi_{3}^3, \dots, \phi_{4}^3$ (notation as in Definition~\ref{Def: Bases}) is 
	\[
		\begin{pmatrix}
			a_{1, 0, 2, 0}	&	0			&	a_{0, 1, 2, 0}	&		0	\\
			a_{1, 0, 1, 1}	&	a_{1, 0, 2, 0}	&	a_{0, 1, 1, 1}	&		a_{0, 1, 2, 0}	\\
			a_{1, 0, 0, 2}	&	a_{1, 0, 1, 1}	&	a_{0, 1, 0, 2}	&		a_{0, 1, 1, 1}	\\
			0			&	a_{1, 0, 0, 2}	&	0			&		a_{0, 1, 0, 2}
		\end{pmatrix}.
	\]
	By definition of the distinguished orientation, the derivative is the class of the determinant of this matrix.  The matrix, however, is the Sylvester matrix of  $\frac{\partial f}{\partial e_1}(x e_{3} + y e_{4})$ and $\frac{\partial f}{\partial e_2}(x e_{3} + y e_{4})$ (considered as polynomials in $\phi_{3}$ and $\phi_{4}$), and so its determinant is the resultant by definition.
\end{proof}

\begin{co} \label{Corollary: LocalIndexIsTypeRestated}
	The type of a line on a smooth cubic surface $V = \{f=0\}$ equals the index of $\sigma_{f}$ at the corresponding zero.  
	
	In particular, the line is hyperbolic if and only if the corresponding index is $\langle 1 \rangle$.
\end{co}
\begin{proof}
	This is a restatement of Lemma~\ref{Lemma: LocalIndexIsType} combined with Proposition~\ref{Prop: ExpressionForType}.
\end{proof}

\begin{co} \label{Corollary: OnlySImpleZeros}
	The vector field $\sigma_{f}$ defined by the equation of a smooth cubic surface has only simple zeros.
	
	More generally, a line on a possibly singular cubic surface $\{ f=0 \}$ that is disjoint from the singular locus corresponds to a simple zero of $\sigma_{f}$.
\end{co}
\begin{proof}
	It is enough to verify the claim after extending scalars to $\kbar$, in which case, with notation as in Lemma~\ref{Lemma: LocalIndexIsType}, it is enough to show that $ \operatorname{Res}( \frac{\partial f}{\partial e_1}( x e_{4} + y e_{4}), \frac{\partial f}{e_2}( x e_{4} + y e_{4}))$ is nonzero.  If not, there is a nonzero vector $v \in k^{\oplus 4}$ such that $\frac{\partial f}{\partial e_1}(v) = \frac{\partial f}{\partial e_2}(v)=0$.  Since $f|S=0$, we also have $\frac{\partial f}{\partial e_3}(v) = \frac{\partial f}{\partial e_4}(v)=0$, so $k \cdot v \subset k^{\oplus 4}$ defines a point of $\bbP^{3}_{k}$ that lies in the singular locus of $V$.  This contradicts the hypothesis that $V$ is smooth along the line.
\end{proof}

\begin{co}\label{co:line_field_def_is_separable}
	The field of definition of a line contained in a smooth cubic surface is a separable extension of $k$.
\end{co}
\begin{proof}
	We conclude from Corollary~\ref{Corollary: OnlySImpleZeros} that the zero locus $\{ \sigma_{f}=0 \}$ is geometrically reduced.  If $L$ is the field of definition of a line, then the natural inclusion $\Spec(L) \to \{ \sigma_{f}=0\}$ is a connected component, so $\Spec(L)$ itself is geometrically reduced or equivalently $L/k$ is separable.  
\end{proof}

We now show that the conditions of Definition~\ref{df:e(E)well-defined} are satisfied so that $e^{\bbA^{1}}(\calE)$ is well-defined.  Recall that we need to show that there are many affine lines in $H^{0}(G, \calE)$ that avoid the locus of sections with nonisolated zeros.  We begin by introducing some schemes related to that locus, such as the universal cubic surface $\calV$ over the moduli space of cubic surfaces $\bbP^{19}_{k}$.

\begin{df}
	Denote the basis of $(k^{\oplus 20})^{\vee}$ dual to the standard basis by $\{ a_{i, j, k, l} \colon i+j+k+l=3 \}$, and define
	\[
		\calV := \left\{ \sum_{i+j+k+l=3} a_{i, j, k, l} x_1^i x_2^j x_3^k x_4^l=0 \right\} \subset \bbP^{19}_{k} \times_{k} \bbP^{3}_{k}.
	\]	
	
	Define
	\begin{gather*}
		\calV_{\text{sing}} \subset \calV \text{ to be the nonsmooth locus of } 	\calV \to \bbP^{19}_{k}, \\
		I_{1} \subset \calV \text{ to be  } \calV_{\text{sing}} \cap \{ \text{Hessian of $f$}=0 \}, \text{ and} \\
		I_{2} \text{ to be the closure of the complement of the diagonal in } \calV_{\text{sing}} \times_{\bbP^{19}_{k}}  \calV_{\text{sing}}.
	\end{gather*}
\end{df}
Informally, $I_{1}$ is the moduli space of pairs consisting of a cubic surface and a singularity of the surface that is worse than a node, and $I_{2}$ is the Zariski closure of the moduli space of triples consisting of a cubic surface and a pair of distinct singularities.

The following dimension estimates of $I_{1}$ and $I_{2}$ will be used to bound the locus of sections with nonisolated zeros.

\begin{lm} \label{Lemma: BoundBadLocus}
	The images of $I_{1}$ and $I_{2}$ under the projections onto  $\bbP^{19}_{k}$ are closed subsets of dimension  $17$.
\end{lm}
\begin{proof}
	In proving the lemma, the key point is to show that $I_{1}$ and $I_{2}$ are irreducible of dimension $17$, and we prove this by analyzing the projections $I_{1} \to \bbP^{3}_{k}$ and $I_{2} \to \bbP^{3}_{k} \times_{k} \bbP^{3}_{k}$.  We can assume $k=\kbar$ since it is enough to prove the lemma after extending scalars.   Consider first the projection $I_{1} \to \bbP^{3}_{k}$.  The fiber of the point corresponding to the subspace  $k \cdot (0,0,0,1)$ is defined by the equations
	\begin{gather*}
		a_{1, 0, 0, 2} =a_{0, 1, 0, 2} = a_{0, 0, 1, 2} = a_{0, 0, 0, 3} = 0, \\
		a_{2,0,0,1} a_{0,1,1,1}^2-a_{1,0,1,1} a_{1,1,0,1} a_{0,1,1,1}+a_{0,2,0,1}
   a_{1,0,1,1}^2+a_{0,0,2,1} a_{1,1,0,1}^2-4 a_{0,0,2,1} a_{0,2,0,1} a_{2,0,0,1}=0,
	\end{gather*}
	and these equations define an  irreducible subvariety of $\bbP^{19}_{k}$ of dimension $19-5=14$.  The same must be true for all other fibers of $I_{1} \to \bbP^{3}_{k}$, so we conclude that $I_{1}$ is irreducible of dimension $14+3=17$.  
	
	The image of $I_{1}$ under the projection $I_{1} \to \bbP^{19}_{k}$ is closed since it is the image of a closed subset under a projective morphism.  The projection is also generically finite.  Indeed, since $I_{1}$ is irreducible, it is enough to exhibit one point in the image with finite preimage, and the preimage of e.g.~the point defined by $x_{4} x_{1}^2+x_{2}^{3}+x_{4}^{3}$ consists of $1$ point.  We conclude that $I_{1}$ and its image have the same dimension, proving the lemma for $I_{1}$.
	
	The proof for  $I_{2}$ proceeds analogously.  In showing that the projection $I_{2} \to \bbP_{k}^{19}$ is generically finite onto its image, replace the polynomial  $x_{4} x_{1}^2+x_{2}^{3}+x_{3}^{3}$ with $x_1 x_2 x_3 + x_1 x_2 x_4 + x_1 x_3 x_4+x_2 x_3 x_4$.
\end{proof}

We now relate $I_{1}$ and $I_{2}$ to the subset of global sections with nonisolated zeros.

\begin{df}
	Define $\calD_0 \subset H^{0}(G, \calE)$ to be the subset of $f$'s such that $\{ f=0 \} \otimes \kbar$ either has a singularity at which the Hessian of $f$ vanishes or has at least two singularities.  (We include $0 \in H^{0}(G, \calE)$ in $\calD_0$.)
\end{df}
In other words, $\calD_{0}$ is the $k$-points of the affine cone over the union of the projections of $I_1$ and $I_{2}$.

The following lemma shows that $\calD_0$ contains the subset of global sections with a nonisolated zero.

\begin{lm} \label{Lemma: NodalHasIsolatedZeros}
	If $f \in H^{0}(G, \calE) - \calD_0$, then $\sigma_{f}$ has only isolated zeros.
\end{lm}
\begin{proof}
	When $\{ f=0 \}$ is  smooth, the claim is Corollary~\ref{Corollary: OnlySImpleZeros}.  Otherwise $\{ f=0 \}$ has a unique singularity at which the Hessian does not vanish.  It is enough to verify the claim after extending scalars to $\kbar$, and after extending scalars and changing coordinates,  we can assume that $k=\kbar$ and $f$ has form
	\[
		f= (x_{1} x_{3}+x_{2}^2)x_{4} + f_3(x_{1}, x_{2}, x_{3}) \text{ for $f_3$ homogeneous of degree $3$.}
	\]
	(Change coordinates so that the singularity is $[0, 0, 0, 1]$.  Then the coefficients of $x_{4}^3$ and $x_{4}^2$ must vanish and the coefficient of $x_{4}$ is a nondegenerate quadratic form, since the Hessian is nonvanishing.  Transform the quadratic form into $x_{1} x_{3}+x_{2}^2$ by a second change of variables.)
	
	Corollary~\ref{Corollary: OnlySImpleZeros} states that the zeros of $\sigma_{f}$  corresponding to lines disjoint from the singular locus are isolated (in fact simple).  The lines that pass through the singular locus are described as follows.  The lines passing through the singularity  (but possibly not lying on $\{ f=0\}$) are in bijection with $\bbP_{k}^{1}(k)$ with the $1$-dimensional subspace  $k \cdot (\alpha, \beta)$ corresponding to the line defined by the  subspace  $k \cdot (0, 0, 0, 1)+k \cdot (\alpha^2, \alpha \beta, \beta^2, 0)$.  Such a line is contained in $\{ f=0 \}$ precisely when $f_{3}(\alpha^2, \alpha \beta,  \beta^2)=0$.  The polynomial $f_3$ is not identically zero (for otherwise $\{ f=0 \}$ has positive dimensional singular locus), so there are at most $6$ lines on $\{ f=0 \}$ that pass through the origin.  In particular, there are only finitely many zeros of $\sigma_{f}$  that correspond to lines that meet the singular locus, so these zeros must be isolated as well.
\end{proof}

We can now show that Euler number of $\calE$ is well-defined.

\begin{lm} \label{Lemma: EulerNumberDefined}
	The vector bundle $\calE$ satisfies the hypotheses to Definition~\ref{df:e(E)well-defined}.
\end{lm}
\begin{proof}
	By Lemma~\ref{Lemma: NodalHasIsolatedZeros}, it is enough to prove that, after possibly passing to an odd degree field extension, any two elements of $H^{0}(G, \calE)-\calD_{0}$ can be connected by affine lines that do not meet $\calD_0$.  We will deduce  the claim from the fact that $\calD_0$ is the $k$-points  of a subvariety of codimension at least $2$ (i.e.~from Lemmas~\ref{Lemma: BoundBadLocus} and \ref{Lemma: NodalHasIsolatedZeros}).   
	
	Let $f, g \in H^{0}(G, \calE)-\calD_0$ be given.  A dimension count shows that, after possibly passing to an odd degree extension when $k$ is finite, there exists a $3$-dimensional subspace $S \subset H^{0}(G, \calE)$ such that $S \cap \calD_0$ is the $k$-points of the cone over a $0$-dimensional subscheme.  In other words, $\calD_0$ is a union of finitely many $1$-dimensional subspaces.  After possibly further passing to a larger odd degree extension of $k$, there are strictly fewer $1$-dimensional subspaces contained in $\calD_0$ than there are $2$-dimensional subspaces of $S$ containing $f$, so we can pick a $2$-dimensional subspace $T_{f} \subset S$ that contains $f$ and is not contained in $\calD_0$ as well as an analogous subspace $T_{g}$ containing $g$.  By another dimension count, the intersection $T_{f} \cap T_{g}$ is nonzero, so we can pick a nonzero vector $h \in T_{f} \cap T_{g}$.  Both the line joining $f$ to $h$ and the line joining $h$ to $g$ are disjoint from $\calD_0$ by construction.
\end{proof}

\begin{proof}[Proof of Main Theorem]
	The $\bbA^1$-Euler number of $\calE$ is well-defined by Lemma~\ref{Lemma: EulerNumberDefined}, and Lemma~\ref{Lemma: LocalIndexIsType} identifies the left-hand side of \eqref{Eqn: FinalLineCount} with the Euler number $e^{\bbA^{1}}(\calE, \sigma_f)$, where $f$ is a degree $3$ homogeneous polynomial whose zero locus is $V$. By Lemma~\ref{Lemma: EulerNumberDefined}, this expression is independent of the choice of smooth surface.  We complete the proof by computing the expression for two especially simple surfaces.
	
	Consider first the surface $V$ defined by $f = x_1^3 + x_2^3 + x_3^3 + x_4^3$ over a field $k$ that has characteristic not equal to $3$ and does not contain a primitive third root of unity $\zeta_{3}$.  This is a smooth surface that contains  $3$ lines with field of definition $k$ and 12 lines with fields of definition $L=k(\zeta_{3})$.  The lines are described as follows.  For $i_1, i_2 = 0, 1, 2$, the subspace
	\begin{equation} \label{Eqn: LineOnFermat}
		k(\zeta_{3}) \cdot (-1, \zeta_{3}^{i_{1}}, 0, 0) + k(\zeta_{3}) \cdot (0, 0, -1, \zeta_{3}^{i_{2}}) \subset L^{\oplus 4}
	\end{equation}
	defines a morphism $\Spec(L) \to G$ with image a line contained in $V$.  Permuting the coordinates of $k^{\oplus 4}$, we obtain $2$ more morphisms.  Varying over all $3^2=9$ choices of $i_1, i_2$, we obtain in this manner $27 = 3 \cdot 9$ morphisms $\Spec(L) \to G$, and the images are the desired lines. Since the weighted count (with weight given by the degree of the field of definition) of these lines is $27$, these lines must be all the lines on $V$.
	
	Every line on $V$ is hyperbolic.  Indeed, a line defined by \eqref{Eqn: LineOnFermat} is hyperbolic by  Proposition~\ref{Prop: ExpressionForType} since $\operatorname{Res}(\frac{\partial f}{\partial e_1}|S, \frac{\partial f}{\partial e_2}|S) = 9$ for $e_{1} = (1, 0, 0, 0), e_{2} = (0, 0, 1, 0)$.  All other lines can be obtained from these lines by acting by automorphisms, so we conclude that all lines are hyperbolic and 
	\begin{align}
		e^{\bbA^{1}}(\calE) 		=& 	3 \cdot \langle 1 \rangle + 12 \cdot \operatorname{Tr}_{k(\zeta_{3}))/k}( \langle 1 \rangle) \notag \\
								=&	3 \cdot \langle 1 \rangle + 12 \cdot (\langle 2\rangle + \langle 2 \cdot (-3)  \rangle). \label{Eqn: FirstExpressionForEuler}
	\end{align}

	If $k$ contains a primitive 3rd root of unity, then the above argument remains valid except all the lines on $V$ are defined over $k$, so
	\begin{equation}
		e^{\bbA^{1}}(\calE) = 27 \cdot \langle 1 \rangle. \label{Eqn: FirstExpressionForEulerAlt}
	\end{equation}
	
	Next we turn our attention to the case where $\operatorname{char} k \ne 5$.  In this case, we consider the smooth cubic surface defined by   $f = \sum\limits_{\substack{i, j=1 \\ i\neq j}}^{4} x_{i}^{2} x_{j} + 2 \sum\limits_{i=1}^{4} x_{1} x_{2} x_{3} x_{4} x_{i}^{-1}$.  (The equation $f$ equals $((x_1+x_2+x_3+x_4)^{3} -x_{1}^3-x_{2}^3-x_{3}^{3}-x_{4}^3)/3$ when $\operatorname{char} k \ne 3$.)  As an aid for analyzing the lines on $V$, we introduce the action of the symmetric group $S_{5}$ on $5$ letters defined by 
	\[
		\sigma(x_{i}) = \begin{cases}
						-x_{1}-x_{2}-x_{3}-x_{4} & \text{if $\sigma(i)=5$;} \\
						x_{\sigma(i)} & \text{ otherwise}
					\end{cases}
	\]
	for $\sigma \in S_{5}$.  This equation leaves $f$ invariant, so it induces an action on $V$.

	Consider first the case where $k$ does not contain $\sqrt{5}$.  The subspaces 
	\begin{gather*}	
		k \cdot (1, -1, 0, 0) + k \cdot (0, 0, 1, -1) \subset k^{\oplus 4}  \text{ and }  \label{Eqn: LineOnClebschOne}  \\
		k(\sqrt{5}) \cdot (2, \alpha, \overline{\alpha}, \overline{\alpha}) + k(\sqrt{5}) \cdot (\alpha, \overline{\alpha}, \overline{\alpha}, \alpha) \subset k(\sqrt{5})^{\oplus 4} \text{ for $\alpha = \frac{-1+\sqrt{5}}{2}$, $\overline{\alpha} = \frac{-1-\sqrt{5}}{2}$.} \label{Eqn: LineOnClebschTwo}
	\end{gather*}
	define lines on $V$ with fields of definition respectively equal to $k$ and $k(\sqrt{5})$.  Computing the type of the first line using the partial derivatives with respect to $(1, 0, 0, 0)$ and $(0, 0, 1, 0)$, we see that the line is hyperbolic, and for the second line, computing with respect to $(0, 1, 0, 0)$ and $(0, 0, 0, 1)$ shows that the type is $-25/2 \cdot (5 + \sqrt{5})$.

	Under the action of $S_{5}$, the orbit of the first line has 15 elements, while the  second has 6 elements, so have found all the lines.  Furthermore, the type is preserved by automorphisms so 
	\begin{align}
	e^{\bbA^{1}}(\calE) 			=&		15 \cdot \langle 1 \rangle + 6 \cdot \operatorname{Tr}_{k(\sqrt{5})/k}( \langle -25/2 \cdot (5 + \sqrt{5}) \rangle) \notag \\
								=&		15 \cdot \langle 1 \rangle + 12 \cdot \langle  -5 \rangle. \label{Eqn: SecondExpressionForEuler}
	\end{align}

	When $k$ contains $\sqrt{5}$, the same argument remains valid except the $6$ lines with field of definition $k(\sqrt{5})$ are replaced by $12$ lines defined over $k$: six with type $-(5+\sqrt{5})/2$ and six with type $-(5-\sqrt{5})/2$.  We deduce
	\begin{equation} \label{Eqn: SecondExpressionForEulerAlt}
		e^{\bbA^{1}}(\calE) = 15 \cdot \langle 1 \rangle + 6  \cdot \langle -(5+\sqrt{5})/2 \rangle + 6  \cdot \langle -(5-\sqrt{5})/2 \rangle.
	\end{equation}
	
	To complete the proof, we need to show that all the classes we just computed equal 
	\begin{equation} \label{Minkowski}
		15 \cdot \langle 1 \rangle + 12 \cdot \langle -1 \rangle.
	\end{equation}
	The expressions in \eqref{Eqn: FirstExpressionForEuler}, \eqref{Eqn: SecondExpressionForEuler}, and \eqref{Minkowski} are all defined over the prime field, so it is enough to show equality when $k = \mathbb{F}_{p}$ or $\mathbb{Q}$.  When $k=\mathbb{F}_{p}$, two elements in $\operatorname{GW}(\mathbb{F}_{p})$ are equal provided they have the same rank and discriminant, and all three classes have rank $27$ and discriminant $1 \in \mathbb{F}^{\ast}_{p}/(\mathbb{F}^{\ast}_{p})^2$, hence are equal.  When $k=\mathbb{Q}$, one can either apply \cite[VI \S 4Theorem 4.1]{lam05} (it is enough to show that \eqref{Eqn: FirstExpressionForEuler}, \eqref{Eqn: SecondExpressionForEuler}, and \eqref{Minkowski} are equal in the Witt group) or note that all three classes have rank $27$, signature $3$, discriminant $1$, and trivial Hasse--Witt invariant, so they are equal by the Hasse--Minkowski theorem.
	
	When $k$ contains a primitive third root of unity $\zeta_{3}$, \eqref{Eqn: FirstExpressionForEulerAlt} equals the class \eqref{Eqn: FirstExpressionForEuler} because $\operatorname{Tr}_{k(\zeta_{3})/k}(\langle 1 \rangle) \otimes_{k} k(\zeta_{3}) = 2 \langle 1 \rangle$ by  \cite[Theorem~6.1]{lam05} (or a direct computation) and similarly with  \eqref{Eqn: SecondExpressionForEuler} and \eqref{Eqn: SecondExpressionForEulerAlt} when $k$ contains $\sqrt{5}$.
	\end{proof}

We now deduce Theorem~\ref{Thm: FiniteFieldLineCount} from the Main Theorem.

\begin{proof}[Proof of Theorem~\ref{Thm: FiniteFieldLineCount}]
	Take the discriminant of  \eqref{Eqn: FinalLineCount}.  The right-hand side is $1 \in \mathbb{F}_{q}^{\ast} / (\mathbb{F}_{q}^{\ast})^{2}$. The left-hand side evaluates to 
\begin{multline*}
		\#\text{elliptic lines on $V$ with field of definition $\mathbb{F}_{q^a}$ for $a$ odd} \\+  \#\text{hyperbolic lines on $V$ with field of definition $\mathbb{F}_{q^a}$ for $a$ even} = 0 \text{ (mod 2)}
\end{multline*}
because 
		\begin{equation} \label{Eqn: FiniteFieldDiscriminant}
		\operatorname{Disc}(\operatorname{Tr}_{\mathbb{F}_{q^a}/\mathbb{F}_{q}}( \langle u \rangle)) = \begin{cases}
																	\text{a square}	& \text{ if $a$ is odd, $u$ is a square;} \\
																	\text{a square}	& \text{ if $a$ is even, $u$ is a nonsquare;} \\
																	\text{a nonsquare}	& \text{ if $a$ is even, $u$ is a square;} \\
																	\text{a nonsquare}	& \text{ if $a$ is odd, $u$ is a nonsquare.} \\
																\end{cases}
		\end{equation}
		Equation~\eqref{Eqn: FiniteFieldDiscriminant} is e.g.~\cite[(II.2.3)~Theorem]{Conner_Perlis}. One argument is to reduce to the case where $u=1$ by showing that 
	\begin{align} \label{Eqn: DiscRelation}
		\operatorname{Disc}(\operatorname{Tr}_{\mathbb{F}_{q^a}/\mathbb{F}_{q}}( \langle u \rangle)) =&  \operatorname{Norm}(u) \cdot \operatorname{Disc}(\operatorname{Tr}_{\mathbb{F}_{q^a}/\mathbb{F}_{q}}( \langle 1 \rangle)) \notag
	\end{align}
	and using Hilbert's Theorem~90 to show that $\operatorname{Norm}(u)$ is a perfect square in $\mathbb{F}_{q^a}$ if and only if $u$ is a perfect square in $\mathbb{F}_{q}$.  When  $u=1$, use the alternative description of the  discriminant as the square of the product of the differences  of Galois conjugates of a primitive element. By Galois theory, this square is a perfect square in $\mathbb{F}_{q}$ if and only if the Frobenius element acts on the conjugates $x$ as an even permutation, which is the case if and only if $a$ is odd because the Frobenius element acts as a cyclic permutation of length $a$.
\end{proof}

}
\bibliographystyle{amsalpha}

%the bibliographystyle was modified by Joe from amsalpha. If I change bibliographystyle,  "thesis" bibliography will still compile, but without the overides to cite EGA.
%
\bibliography{CubicSurface}

\end{document}